\documentclass[11pt]{amsart}
\title{The modal theory of the category of sets}
\author{Wojciech Aleksander Wo\l oszyn}
\address[Wojciech Aleksander Wo\l oszyn]
{Mathematical Institute, University of Oxford, Andrew Wiles Building, Radcliffe Observatory Quarter, Woodstock Road, Oxford, OX2 6GG, United Kingdom \&\ St Hilda's College, Cowley Place, Oxford, OX4 1DY, United Kingdom}
\email{wojciech@woloszyn.org}
\urladdr{https://woloszyn.org}

\usepackage{latexsym}   
\usepackage{amsfonts}   
\usepackage{amsmath}    
\usepackage{amssymb}    
\usepackage{mathrsfs}   

\usepackage[hidelinks]{hyperref} 
\usepackage[utf8]{inputenc} 
\usepackage{microtype} 

\usepackage[rgb,dvipsnames]{xcolor} 
\usepackage{tikz} 
\usetikzlibrary{backgrounds} 
\usetikzlibrary{arrows,arrows.meta,petri,topaths,positioning,shapes,shapes.misc,patterns,calc,decorations,decorations.pathreplacing,hobby} 
\usepackage{tikz-cd} 

\usepackage{adjustbox} 
\usepackage{stackengine} 
\usepackage{caption} 
\usepackage{floatrow} 
\usepackage{float} 
\usepackage[inline]{enumitem} 
\usepackage{tasks} 
\usepackage{booktabs} 
\usepackage[backend=biber,style=alphabetic,maxbibnames=15,maxcitenames=6,dateabbrev=false,autolang=hyphen]{biblatex}

\usepackage{mathtools} 

\usepackage{MathMacrosWAW} 

\renewcommand{\UrlFont}{\sffamily} 
\renewbibmacro{in:}{\ifentrytype{article}{}{\printtext{\bibstring{in}\intitlepunct}}} 
\DeclareFieldFormat{url}{{\UrlFont\url{#1}}} 
\DeclareFieldFormat{urldate}{
  (version \thefield{urlday}\addspace%
  \mkbibmonth{\thefield{urlmonth}}\addspace%
  \thefield{urlyear}\isdot)}
\DeclareFieldFormat{eprint:arxiv}{
\ifhyperref
    {\href{http://arxiv.org/abs/#1}{%
    arXiv\addcolon\nolinkurl{#1}}\iffieldundef{eprintclass}{}{\UrlFont{\mkbibbrackets{\thefield{eprintclass}}}}}
    {arXiv\addcolon\nolinkurl{#1}\iffieldundef{eprintclass}{}{\UrlFont{\mkbibbrackets{\thefield{eprintclass}}}}}}
\bibliography{HamkinsBiblio,MathBiblio,WoloszynBiblio}
\renewcommand\emptyset{\varnothing}
\newcommand\Val{\mathord{\rm Val}}

\tikzset{>=Stealth,
  dot/.style={circle,draw,fill,inner sep=#1},
    dot/.default=.7pt
}

\makeatletter
\def\@tocheadstart{}
\def\@tocheadend{}

\makeatother

\begin{document}
\raggedbottom

\begin{abstract}
	We classify the propositional modal validities of the category of sets under its natural classes of morphisms, with the answer depending on the morphisms, the size of the world, and the allowed substitution instances. Our main technical tool is a modality-and-quantifier elimination theorem for the first-order modal language of equality, reducing formulas to finite Boolean combinations of partition conditions and exact-cardinality assertions. From this we obtain exact classifications for the principal categories of sets, including the full subcategories of finite sets as well as infinite sets. In particular, finite $n$-element worlds in $\Sets$ with parameters realize $\theoryf{Prepartition}_n$, finite $n$-element worlds in $\Sets[$\onto$]$ realize $\theoryf{Grz.3J}_n$ sententially, and in the infinite-only subcategories sentential validities collapse to $\theoryf{Triv}$, while the formulaic validities for functions and surjections are exactly $\theoryf{Grz.2}$.
\end{abstract}
\maketitle
\vspace*{-1.6em}

We determine the propositional modal validities of the category of sets under its most natural classes of morphisms. The answer depends on the morphisms, the size of the world, and the allowed substitution instances, with companion results for the corresponding full subcategories of finite sets and infinite sets. The resulting classification is summarized in the~\nameref{Maintheorem.Sets}, below.

\begin{maintheorem*}\makeatletter\def\@currentlabelname{main theorem}\makeatother\label{Maintheorem.Sets}
    The following table summarizes the propositional modal validities considered in this paper. For each class of morphisms, the validities at any finite world are the same in $\Sets$ as in $\FinSets$.
    \renewcommand{\thefootnote}{\fnsymbol{footnote}}

    \newcommand{\myref}[1]{\textsuperscript{\supertiny\color{SalmonSunrise}\ref{#1}}}
    \newcommand{\theoryref}[2]{\theoryf{#1}\myref{#2}}

    \begin{table}[H]
        \small
        \begin{adjustbox}{center}
            \begin{tabular}{@{}llllll@{}}
                \toprule
                                       & \multicolumn{5}{c}{Propositional modal validities} \\
                \cmidrule(l){2-6}
                                       & Empty world & \multicolumn{2}{c}{Finite worlds of size $n>0$} & \multicolumn{2}{c}{Infinite worlds} \\\cmidrule(l){2-6}
                Category $\Sets$ with & --- & Sentential & Formulaic & Sentential & Formulaic \\\midrule
                $\to$ Functions    & \theoryref{Lollipop}{Theorem.Sets-empty-worlds-sentences-Lollipop} & \theoryref{S5}{Theorem.Sets-non-empty-worlds-sentences-S5} & $\theoryf{Prepartition}_n$\myref{Theorem.Sets-world-size-n-formulas-Prepartition} & \theoryref{S5}{Theorem.Sets-non-empty-worlds-sentences-S5} & \theoryref{S4.2}{Theorem.Sets-infinite-worlds-formulas-S4.2} \\\midrule
                $\onto$ Surjections   & \theoryref{Triv}{Theorem.Surj-singleton-or-empty-world-Triv} & $\theoryf{Grz.3J}_n$\myref{Theorem.Surj-world-size-n-sentences-Grz.3Jn} & $\theoryf{Partition}_n$\myref{Theorem.Surj-world-size-n-formulas-Partition} & \theoryref{Grz.3}{Theorem.Surj-infinite-worlds-Grz.3} & \theoryref{Grz.2}{Theorem.Surj-infinite-worlds-formulas-Grz.2} \\\midrule
                $\into$ Injections  & \theoryref{Grz.3}{Theorem.Inj-finite-worlds-Grz.3} & \theoryref{Grz.3}{Theorem.Inj-finite-worlds-Grz.3} & \theoryref{Grz.3}{Theorem.Inj-finite-worlds-Grz.3} & \theoryref{Triv}{Theorem.Inj-infinite-world-parameters-Triv} & \theoryref{Triv}{Theorem.Inj-infinite-world-parameters-Triv} \\\midrule
                $\incl$ Inclusions     & \theoryref{Grz.3}{Observation.Incl-and-Inj-same} & \theoryref{Grz.3}{Observation.Incl-and-Inj-same} & \theoryref{Grz.3}{Observation.Incl-and-Inj-same} & \theoryref{Triv}{Observation.Incl-and-Inj-same} & \theoryref{Triv}{Observation.Incl-and-Inj-same} \\\midrule
                $\bij$ Bijections    & \theoryref{Triv}{Observation.Iso-trivializations} & \theoryref{Triv}{Observation.Iso-trivializations} & \theoryref{Triv}{Observation.Iso-trivializations} & \theoryref{Triv}{Observation.Iso-trivializations} & \theoryref{Triv}{Observation.Iso-trivializations} \\\midrule
                $\ident$ Identities    & \theoryref{Triv}{Observation.Iso-trivializations} & \theoryref{Triv}{Observation.Iso-trivializations} & \theoryref{Triv}{Observation.Iso-trivializations} & \theoryref{Triv}{Observation.Iso-trivializations} & \theoryref{Triv}{Observation.Iso-trivializations} \\
                \midrule
            \end{tabular}
        \end{adjustbox}
        \refstepcounter{table}
\label{Table.Sets-propositional-modal-validities}
    \end{table}

    For the corresponding full subcategories of infinite sets, sentential validities are always \theoryref{Triv}{Corollary.InfSets-sentences-Triv}. Formulaically, $\InfSets$ and $\InfSets[$\onto$]$ both validate \theoryf{Grz.2}\textsuperscript{\supertiny\color{SalmonSunrise}
\ref{Theorem.InfSets-formulas-Grz.2},\ref{Theorem.InfSurj-infinite-worlds-formulas-Grz.2}}; $\InfSets[$\into$]$ and $\InfSets[$\incl$]$ validate
\theoryf{Triv}\textsuperscript{\supertiny\color{SalmonSunrise}
\ref{Theorem.Inj-infinite-world-parameters-Triv},\ref{Observation.Incl-and-Inj-same}}; and $\InfSets[$\bij$]$ and $\InfSets[$\ident$]$ validate \theoryref{Triv}{Observation.Iso-trivializations}.
\end{maintheorem*}

The table and the accompanying remarks record how the modal validities change with the morphisms, the size of the world, and the substitution regime. For example, in $\Sets$ nonempty worlds validate $\theoryf{S5}$ sententially, while finite $n$-element worlds validate $\theoryf{Prepartition}_n$ formulaically; in $\Sets[$\onto$]$, finite $n$-element worlds validate $\theoryf{Grz.3J}_n$ sententially and $\theoryf{Partition}_n$ formulaically; while among the infinite-set subcategories the sentential validities always collapse to $\theoryf{Triv}$, with the formulaic validities given by $\theoryf{Grz.2}$ for functions and surjections and by $\theoryf{Triv}$ for injections, inclusions, bijections, and identities. The superscript references attached to each cited theory point to the corresponding proof.

The technical backbone of these calculations is the~\nameref{Modalityeliminationtheorem.Sets} for the first-order modal language of equality. In the relevant categories of sets, every formula reduces to a finite Boolean combination of equality patterns and exact-cardinality assertions. This reduction isolates the concrete combinatorial structures that govern the later modal logics: partition lattices and prepartition prelattices in the parameterized cases, and finite linear orders of cardinality types in the sentential surjective case.

This paper belongs to modal model theory~\cite{HamkinsWoloszyn:Modal-model-theory}, which studies a mathematical structure within a class of similar structures under an extension concept giving rise to natural notions of possibility, necessity, and modal validity. In the principal cases considered there, one works with classes such as $\Mod(T)$ under the submodel relation, thereby extending the potentialist perspective developed in set-theoretic and arithmetic potentialism~\cite{HamkinsLinnebo:Modal-logic-of-set-theoretic-potentialism,Hamkins:The-modal-logic-of-arithmetic-potentialism}. Related analyses include the modal logic of abelian groups~\cite{Berger-Block-Lowe:The-modal-logic-of-abelian-groups}. Our contribution is to generalize modal model theory from such extension-based systems to arbitrary concrete categories of structures in a common first-order language, allowing arbitrary morphisms rather than only inflationary substructure-type maps, and to give an exact classification for the category of sets under its natural classes of morphisms, together with companion results for the full subcategories of finite sets and infinite sets.

Sections~\ref{Section.Modal-semantics}--\ref{Section.validity} isolate the minimal general framework and modal tools used later. Section~\ref{Section.Modality-elimination-in-Sets} proves the elimination theorem for sets. The remaining sections analyze, in turn, the cases of functions, surjections, injections and inclusions, and finally bijections and identities.

\section{Modal semantics for a concrete category}\label{Section.Modal-semantics}

Potentialist systems, as studied in set-theoretic and arithmetic potentialism, are reflexive transitive Kripke systems of structures in which accessibility is inflationary on the underlying domains, so that individuals persist into accessible worlds \cite{HamkinsLinnebo:Modal-logic-of-set-theoretic-potentialism,Hamkins:The-modal-logic-of-arithmetic-potentialism}. More recently, potentialist systems were applied to arbitrary classes of structures in modal model theory~\cite{HamkinsWoloszyn:Modal-model-theory}. The present framework of Kripke categories generalizes that setting by allowing arbitrary concrete categories of structures and by dropping the requirement that accessibility mappings be inclusions, or even inflationary on the underlying domains.\footnote{For philosophical background on the distinction between actual and potential infinity, see \cite{LinneboShapiro2017:Actual-and-potential-infinity}. For the modern mathematical development of potentialist systems and their modal logic, see \cite{HamkinsLinnebo:Modal-logic-of-set-theoretic-potentialism,Hamkins:The-modal-logic-of-arithmetic-potentialism,HamkinsWoloszyn:Modal-model-theory}.} Recall that an \emph{$\commonL$-homomorphism} between $\commonL$-structures $U$ and $V$ is a function $h \colon U \to V$ that preserves atomic truths: whenever $\varphi$ is $\commonL$-atomic and $U\satisfies\varphi[\nu]$, we have $V\satisfies\varphi[h\compose\nu]$.

\begin{maindefinition*}\makeatletter\def\@currentlabelname{main definition}\makeatother\label{Definition.Kripke-category}
	A \emph{Kripke category} is a concrete category of $\commonL$-structures in a common first-order language $\commonL$. That is, the objects in the category are each $\commonL$-structures, and the morphisms are $\commonL$-homomorphisms. We shall call these objects \emph{worlds} and morphisms \emph{accessibility mappings}.
\end{maindefinition*}

Typical examples are the categories of groups, graphs, and sets with their usual homomorphisms; our main case is $\Sets$ in the language of equality.

Let $\LDiamond$ denote the first-order modal expansion of $\commonL$. If $W_0$ is a world in a Kripke category $\KripkeCat$, write $W_0^{\uparrow}$ for the class of worlds accessible from $W_0$, and let $\mathrm{Cone}(W_0)$ be the full subcategory of $\KripkeCat$ on $W_0^{\uparrow}$. Thus $\mathrm{Cone}(W_0)$ is the part of the category visible from $W_0$; categorically, there is a canonical forgetful functor from the coslice category $(W_0/\KripkeCat)$ to $\mathrm{Cone}(W_0)$.

The satisfaction relation $W \satisfies_{\KripkeCat} \varphi[\nu]$ for worlds $W$ on the cone above $W_0$ is defined by the usual Tarskian clauses together with:
\begin{enumerate}[partopsep=5pt]
	\setlength\itemsep{5pt}
	\item $W \satisfies_{\KripkeCat} \possible \varphi[\nu]$ if there is an accessibility mapping $f \colon W \to U$ in $\mathrm{Cone}(W_0)$ such that $U \satisfies_{\KripkeCat} \varphi[f \compose \nu]$;
	\item $W \satisfies_{\KripkeCat} \necessary \varphi[\nu]$ if for every accessibility mapping $f \colon W \to U$ in $\mathrm{Cone}(W_0)$, we have that $U \satisfies_{\KripkeCat} \varphi[f \compose \nu]$.
\end{enumerate}
An $\LDiamond$-assertion is \emph{true} at $W$ when $W \satisfies_{\KripkeCat} \varphi$. When the ambient category is clear, we simply write $\satisfies$.

\section{Elementary modal model theory}

The basic modal model theory of Kripke categories closely parallels the potentialist setting. We record here only the facts needed later.

\newtheorem*{renaminglemma}{Renaming lemma}%
\begin{renaminglemma}\makeatletter\def\@currentlabelname{renaming lemma}\makeatother\label{Renaming-lemma}
	Suppose $\KripkeCat$ is a Kripke category of $\commonL$-structures. Isomorphisms preserve all modal truths in $\KripkeCat$. More precisely, if $W$ and $U$ are worlds in $\KripkeCat$ and $\pi \colon W \cong U$ is an isomorphism in $\KripkeCat$, then for any $\LDiamond$-assertion $\varphi$, we have that
	\begin{equation*}
		W\satisfies\varphi[\nu]\qquad\text{if and only if}\qquad U\satisfies\varphi[\pi \compose \nu].
	\end{equation*}
	In other words, $\LDiamond$-truths are preserved in $\KripkeCat$ by $\commonL$-isomorphisms that belong to the category.
\end{renaminglemma}

\begin{proof}
	Let $W$ and $U$ be worlds in $\KripkeCat$, and let $\pi\colon W\to U$ be an isomorphism.
	We show, by induction on $\varphi$, that for every assignment $\nu$ (with domain containing all free variables of $\varphi$),
	\begin{equation*}
		W\satisfies\varphi[\nu]\quad\Longleftrightarrow\quad U\satisfies\varphi[\pi\compose\nu].
	\end{equation*}

	For atomic $\commonL$-formulae this holds because $\pi$ is an isomorphism of $\commonL$-structures.

	Boolean connectives and quantifiers are handled in the usual way. For example, for $\exists x\,\psi$:
	\begin{align*}
		W\satisfies\exists x\,\psi[\nu]
		&\Longleftrightarrow
		(\exists w\in W)\;W\satisfies\psi[\nu[x\mapsto w]]\\
		&\Longleftrightarrow
		(\exists w\in W)\;U\satisfies\psi[\pi\compose\nu[x\mapsto w]]\\
		&\Longleftrightarrow
		(\exists u\in U)\;U\satisfies\psi[(\pi\compose\nu)[x\mapsto u]].
	\end{align*}
	using the induction hypothesis in the middle step and the bijectivity of $\pi$ in the last.

	For $\possible\psi$, assume that $W\satisfies\possible\psi[\nu]$.
	Then there is an accessibility mapping $f\colon W\to W^*$ with $W^*\satisfies\psi[f\compose\nu]$.
	Since $\pi^{-1}\colon U\to W$ is also a morphism, $f\compose\pi^{-1}\colon U\to W^*$ is an accessibility mapping as well.
	Moreover,
	\begin{equation*}
		(f\compose\pi^{-1})\compose(\pi\compose\nu)=f\compose\nu,
	\end{equation*}
	so $W^*\satisfies\psi[(f\compose\pi^{-1})\compose(\pi\compose\nu)]$, and hence $U\satisfies\possible\psi[\pi\compose\nu]$.
	The reverse direction is analogous (using $\pi^{-1}$ in place of $\pi$), and the case of $\necessary\psi$ follows by duality.

	Thus $W$ and $U$ have the same $\LDiamond$-theory up to renaming of elements along $\pi$.
\end{proof}

A Kripke category $\KripkeCat$ admits \emph{modality trivialization} over a language if every formula $\varphi(\bar x)$ in that language is equivalent in $\KripkeCat$ to $\possible\varphi(\bar x)$, and hence also to $\necessary\varphi(\bar x)$. It is \emph{model complete} if every accessibility mapping $f\colon W\to U$ preserves $\commonL$-truths, that is,
\begin{equation*}
	W\satisfies\varphi[\nu]\qquad\text{if and only if}\qquad U\satisfies\varphi[f\compose\nu],
\end{equation*}
for every $\commonL$-formula $\varphi$. It is enough to require only the forward implication $W\satisfies\varphi[\nu]\Rightarrow U\satisfies\varphi[f\compose\nu]$, since applying it to $\neg\varphi$ yields the reverse implication.

\begin{theorem}\label{Theorem.world-complete-iff-modalities-trivialize}
	A Kripke category $\KripkeCat$ of $\commonL$-structures admits modality trivialization over its modal expansion if and only if it is model complete. That is, $\KripkeCat \satisfies \forall \bar x \, \possible \varphi(\bar x) \iff \varphi(\bar x)$, and equivalently $\KripkeCat \satisfies \forall \bar x \, \necessary \varphi(\bar x) \iff \varphi(\bar x)$, for every formula $\varphi$ in the modal expansion $\LDiamond$ if and only if every accessibility mapping $f \colon W \to U$ in $\KripkeCat$ is $\commonL$-truth preserving.
\end{theorem}
\begin{proof}
	First assume that $\KripkeCat$ admits modality trivialization.
	Let $f\colon W\to U$ be any accessibility mapping and let $\varphi$ be a $\commonL$-formula.
	For any assignment $\nu$, if $W\satisfies\varphi[\nu]$ then, by trivialization, $W\satisfies\necessary\varphi[\nu]$, and hence $U\satisfies\varphi[f\compose\nu]$ by the semantics of $\necessary$.
	Conversely, if $U\satisfies\varphi[f\compose\nu]$, then $f$ witnesses $W\satisfies\possible\varphi[\nu]$, and trivialization yields $W\satisfies\varphi[\nu]$.
	Thus $f$ preserves and reflects $\commonL$-truths, and $\KripkeCat$ is model complete.

	Conversely, assume that $\KripkeCat$ is model complete.
	We first observe that if $\alpha$ is a modality-free ($\commonL$) formula, then $\possible\alpha$ is equivalent to $\alpha$ in $\KripkeCat$.
	Indeed, for any world $W$ and assignment $\nu$, if $W\satisfies\alpha[\nu]$ then the identity mapping witnesses $W\satisfies\possible\alpha[\nu]$.
	For the other direction, if $W\satisfies\possible\alpha[\nu]$ then there is an accessibility mapping $f\colon W\to U$ with $U\satisfies\alpha[f\compose\nu]$; by model completeness, $W\satisfies\alpha[\nu]$.
	By duality, also $\necessary\alpha$ is equivalent to $\alpha$ for modality-free $\alpha$.

	Now define, for each $\LDiamond$-formula $\varphi$, its \emph{modality-erasure} $\varphi^{*}$ recursively by pushing through boolean connectives and quantifiers and by setting
\begin{equation*}
		(\possible\psi)^{*}=\psi^{*}
		\qquad\text{and}\qquad
		(\necessary\psi)^{*}=\psi^{*}.
\end{equation*}
	Then $\varphi^{*}$ is modality-free.
	We claim that $\varphi$ is equivalent to $\varphi^{*}$ in $\KripkeCat$, proved by structural induction on $\varphi$.
	The boolean and quantifier cases are routine.
	For the modal case $\varphi=\possible\psi$, the induction hypothesis gives $\psi\equiv\psi^{*}$, hence $\possible\psi\equiv\possible\psi^{*}$; since $\psi^{*}$ is modality-free, $\possible\psi^{*}\equiv\psi^{*}$ by the previous paragraph, and so $\possible\psi\equiv\psi^{*}=\varphi^{*}$.
	The case $\varphi=\necessary\psi$ is similar (using duality instead of the $\possible$-argument).

	Finally, for any $\varphi$ we have
\begin{equation*}
		\possible\varphi\;\equiv\;\possible\varphi^{*}\;\equiv\;\varphi^{*}\;\equiv\;\varphi,
\end{equation*}
	and by duality also $\necessary\varphi\equiv\varphi$.
	Hence $\KripkeCat$ admits modality trivialization.
\end{proof}

Quantifier elimination, modality elimination, and simultaneous elimination down to a chosen class of formulas are defined exactly as in ordinary model theory, now with equivalence taken in the Kripke category rather than in a theory. In particular, a Kripke category admits \emph{modality elimination} over a language if every formula is equivalent to a modality-free one, and admits simultaneous modality and quantifier elimination \emph{down to} a class $\Lambda$ if every formula is equivalent to an assertion from $\Lambda$. For the categories of sets considered later, the relevant target class will be finite Boolean combinations of equality patterns together with exact-cardinality sentences; section~\ref{Section.Modality-elimination-in-Sets} establishes this explicitly.
\section{Propositional modal theory}\label{Section.Propositional-modal-theory}

We use only a small amount of standard propositional modal logic; for background see~\cite{ChagrovZakharyaschev1997:ModalLogic},~\cite{BlackburnDeRijkeVenema2001:ModalLogic}, or~\cite{Kracht1999:Tools-and-Techniques-in-Modal-Logic}. The propositional modal language has propositional variables, Boolean connectives, and the modal operators $\possible$ and $\necessary$. For $n\geq 1$, define recursively \axiomf{$J_1$} as $\possible\necessary p_1\implies p_1$ and \axiomf{$J_{n+1}$} as $\possible(\necessary p_{n+1} \land \neg J_n)\implies p_{n+1}$. As is standard in the potentialist literature, we use \axiomf{5} in the maximality-principle form $\possible\necessary p\implies p$, so that \axiomf{5} and \axiomf{$J_1$} coincide. We shall use the following familiar axioms:
\begingroup\small
\setlength{\abovedisplayskip}{4pt}
\setlength{\belowdisplayskip}{4pt}
\setlength{\abovedisplayshortskip}{4pt}
\setlength{\belowdisplayshortskip}{4pt}
$$\begin{array}{rl}
		\axiomf{K}       & \necessary( p \implies  q)\implies(\necessary p \implies\necessary q)                \\
		\axiomf{Dual}    & \neg\possible  p \iff \necessary\neg p                                               \\
		\axiomf{4}       & \necessary p \implies \necessary\necessary p                                         \\
		\axiomf{T}       & \necessary p \implies p                                                              \\
		\axiomf{Grz}     & \necessary(\necessary( p \implies\necessary p )\implies p )\implies p                \\
		\axiomf{.2}      & \possible\necessary p \implies\necessary\possible p                                  \\
		\axiomf{.3}      & \necessary (\necessary  p  \implies  q) \lor \necessary (\necessary  q \implies  p ) \\
		\axiomf{J}_{n+1} & \possible(\necessary p_{n+1} \land \neg J_n)\implies  p_{n+1}                        \\
		\axiomf{J_1}     & \possible\necessary p \implies p                                                     \\
		\axiomf{5}       & \possible\necessary p \implies p                                                     \\
		\axiomf{Triv}    & \necessary  p  \iff  p.                                                              \\
	\end{array}$$
\endgroup

A propositional Kripke model is a list of propositional worlds together with an accessibility relation; a frame is the corresponding relational structure without the valuation. We use the standard Kripke semantics. A cluster is a maximal set of mutually accessible worlds, and a chain of length $n$ is a sequence of $n$ nodes lying in successive clusters. Thus a frame has \emph{height} $n$ if its longest strict chain has length $n$.

A \emph{normal} propositional modal theory contains all propositional tautologies and the axioms \axiomf{K} and \axiomf{Dual}, and is closed under modus ponens, substitution, and necessitation. We write $\theoryf{K}$ for the smallest such theory, and for theories $T_0,T_1$ and an axiom $\alpha$ we write $T_0+T_1+\alpha$ for their normal closure. In particular:
{\setlength{\abovedisplayskip}{4pt}\setlength{\belowdisplayskip}{4pt}\setlength{\abovedisplayshortskip}{4pt}\setlength{\belowdisplayshortskip}{4pt}%
\renewcommand{\arraystretch}{0.94}
$$\begin{array}{rcl}
		\theoryf{Triv}     & = & \axiomf{Triv} + \axiomf{4} = \theoryf{S5}+\theoryf{Grz} = \theoryf{Grz.3J}_1 \\
		\theoryf{S5}       & = & \theoryf{S4}+5=\theoryf{S4}+\axiomf{J}_1                                     \\
		\theoryf{S4.3J}_n  & = & \theoryf{S4.3}+\axiomf{J}_n                                                  \\
		\theoryf{S4.3}     & = & \theoryf{S4}+\axiomf{.3}                                                     \\
		\theoryf{S4.2}     & = & \theoryf{S4}+\axiomf{.2}                                                     \\
		\theoryf{Grz.3J}_n & = & \theoryf{Grz.3}+\axiomf{J}_n                                                 \\
		\theoryf{Grz.3}    & = & \theoryf{S4.3}+\axiomf{Grz}                                                  \\
		\theoryf{Grz.2}    & = & \theoryf{S4.2}+\axiomf{Grz}                                                  \\
		\theoryf{Grz}      & = & \theoryf{K}+\axiomf{Grz}=\theoryf{S4}+\axiomf{Grz}                           \\
		\theoryf{S4}       & = & \theoryf{K}+\axiomf{T}+\axiomf{4}.
	\end{array}$$}

\tikzset{
	symbol/.style={
			draw=none,
			every to/.append style={
					edge node={node [sloped, allow upside down, auto=false]{$#1$}}}
		}
}
\makeatletter
\DeclareRobustCommand{\rvdots}{%
	\vbox{
		\baselineskip4\p@\lineskiplimit\z@
		\kern-\p@
		\hbox{.}\hbox{.}\hbox{.}
	}}
\makeatother

The diagram below indicates the relative strength of the modal theories that occur later in the paper; the theories \theoryf{Lollipop}, $\theoryf{Prepartition}_n$, and $\theoryf{Partition}_n$ are introduced in sections~\ref{Section.category-of-sets-with-functions} and~\ref{Section.category-of-sets-with-surjections}.
\begingroup\setlength{\intextsep}{4pt}
\begin{figure}[H]
	\centering
	\begin{tikzcd}[column sep=tiny,row sep=1.0ex]
			& \theoryf{Triv} \arrow[d] \arrow[rd]              &                                           &                                         \\
			& \theoryf{S5} \arrow[rd] \arrow[lddddd]          & \theoryf{Grz.3J}_2 \arrow[rdd] \arrow[d] &                                         \\
			&                                                  & \theoryf{Lollipop} \arrow[d] \arrow[lldddd] &                                      \\
			&                                                  & \theoryf{S4.3J}_2 \arrow[d]               & \theoryf{Grz.3J}_3 \arrow[d] \arrow[ld] \\
			&                                                  & \theoryf{S4.3J}_{3} \arrow[d] \arrow[ld] & \rvdots \arrow[d]                     \\
			& \theoryf{Partition}_3 \arrow[d] \arrow[ld]      & \rvdots \arrow[d]                        & \theoryf{Grz.3J}_n \arrow[d] \arrow[ld] \\
			\theoryf{Prepartition}_3 \arrow[d] & \rvdots \arrow[d]                        & \theoryf{S4.3J}_{n} \arrow[dd] \arrow[ld] & \theoryf{Grz.3} \arrow[dd] \arrow[ldd] \\
			\rvdots \arrow[d]                  & \theoryf{Partition}_n \arrow[rrd] \arrow[ld] &                                       &                                         \\
			\theoryf{Prepartition}_n \arrow[rd] &                                                 & \theoryf{S4.3} \arrow[ld]               & \theoryf{Grz.2} \arrow[ld] \arrow[lld] \\
			& \theoryf{S4.2} \arrow[d]                        & \theoryf{Grz} \arrow[ld]                 &                                         \\
			& \theoryf{S4}                                    &                                           &
		\end{tikzcd}
\end{figure}\endgroup

We shall use the following frame characterizations, all with respect to finite frames: \theoryf{Triv} is characterized by isolated nodes, \theoryf{S5} by isolated clusters, \theoryf{Lollipop} by finite lollipop frames, $\theoryf{S4.3J}_n$ by finite linear preorders of length at most $n$, \theoryf{S4.3} by finite linear preorders, $\theoryf{Prepartition}_n$ by finite partition prelattices over an $n$-element set, \theoryf{S4.2} by finite directed preorders, $\theoryf{Grz.3J}_n$ by finite linear orders of length at most $n$, \theoryf{Grz.3} by finite linear orders, $\theoryf{Partition}_n$ by the finite partition lattice of an $n$-element set, \theoryf{Grz.2} by finite directed partial orders, \theoryf{Grz} by finite partial orders, and \theoryf{S4} by finite preorders. The finite reductions for lollipop frames and partition prelattices are included in the corresponding sections below; the remaining completeness and frame-characterization facts are standard, see the references cited above.
\section{Propositional modal validities in a Kripke category}\label{Section.validity}

A central task in modal model theory is to determine which propositional modal principles are valid in a system of structures or at a particular world. Following the potentialist literature \cite{HamkinsLinnebo:Modal-logic-of-set-theoretic-potentialism,Hamkins:The-modal-logic-of-arithmetic-potentialism} and~\cite{HamkinsWoloszyn:Modal-model-theory}, we extend this notion from potentialist systems to arbitrary Kripke categories.

\begin{keydefinition*}\makeatletter\def\@currentlabelname{key definition}\makeatother\label{Definition.Valid}
	A propositional modal assertion $\varphi(p_0,\ldots,p_n)$, with propositional variables $p_0,\ldots,p_n$, is \emph{valid} at a world $W$ in a Kripke category $\KripkeCat$ with respect to a language of substitution instances when all substitution instances $\varphi(\psi_0,\ldots,\psi_n)$, arising for assertions $\psi_0,\ldots,\psi_n$ in that language, hold true at $W$ in $\KripkeCat$. A propositional modal theory is \emph{valid} at $W$ with respect to that language if every assertion in the theory is valid there.
\end{keydefinition*}

We shall consider, for example, sentential substitutions from $\Lofeq$, formulaic substitutions from $\LDiamondofeq$, and these languages with parameters allowed from the world of evaluation. Let $\Val_{\KripkeCat}(W,\mathscr L)$ denote the set of propositional modal validities at $W$ with respect to assertions from $\mathscr L$. Enlarging the substitution language can only shrink the set of validities:
\begin{equation*}
	\mathscr L \of \mathscr L^{+} \qquad \text{implies} \qquad \Val_{\KripkeCat}(W,\mathscr L^{+}) \of \Val_{\KripkeCat}(W,\mathscr L).
\end{equation*}
In particular, one has the inclusions from formulaic-with-parameters to sentential-with-parameters and from those to their parameter-free counterparts.

As in~\cite[section~5]{HamkinsWoloszyn:Modal-model-theory}, it is important to distinguish propositional modal \emph{theories} from proof systems: in the predicate-modal context with parameters, validities at a world need not be closed under necessitation.

\newtheorem*{ConeLemma}{Cone Lemma}%
\begin{ConeLemma}\makeatletter\def\@currentlabelname{cone lemma}\makeatother\label{ConeLemma}
    Suppose $W_0$ is a world in a Kripke category $\KripkeCat$ of $\commonL$-structures. For a propositional modal theory $T$ that is closed under modal propositional substitutions, the normal closure of $T$ is valid at $W_0$ with respect to a given language of substitution instances if and only if $T$ is true at every world on the cone above $W_0$ in $\KripkeCat$.
\end{ConeLemma}
\begin{proof}
	$(\Rightarrow)$ Suppose that the normal closure of $T$ is valid at $W_0$ (with respect to the given substitution language), and let $W$ be any world on the cone above $W_0$.
	Fix $\sigma\in T$ and any substitution instance $\sigma(\bar\psi)$ in the chosen language.
	Since the normal closure contains $\sigma$ and is closed under necessitation, it contains $\necessary\sigma$, and therefore $W_0\satisfies\necessary\sigma(\bar\psi)$.
	By the semantics of $\necessary$, this implies $W\satisfies\sigma(\bar\psi)$.
	As $W$ was arbitrary, $T$ is true at every world on the cone above $W_0$.

	$(\Leftarrow)$ Conversely, assume that $T$ is true at every world on the cone above $W_0$ (for every substitution instance in the given language).
	Let $S$ be the set of all propositional modal assertions $\chi$ such that every substitution instance of $\chi$ (in the given language) is true at every world on that cone.
	Then $S$ contains $T$, and it also contains all tautologies as well as the axioms $\axiomf{K}$ and $\axiomf{Dual}$, since these are valid in every Kripke semantics.
	Moreover, $S$ is closed under modus ponens and necessitation. Indeed, closure under modus ponens is immediate. For necessitation, if every instance of $\chi$ is true at every world on the cone, then for any cone-world $W$ and any accessibility mapping $f\colon W\to U$, $U$ is again on the cone, so every instance of $\chi$ holds at $U$, and hence every instance of $\necessary\chi$ holds at $W$.

	Now let $\theta$ be any formula in the normal closure of $T$, and fix a substitution instance $\theta(\bar\psi)$ in the given language.
	Choose a Hilbert derivation of $\theta$ from propositional tautologies, the axioms $\axiomf{K}$ and $\axiomf{Dual}$, and axioms from $T$, using only modus ponens and necessitation.
	Apply the substitution $\bar p\mapsto\bar\psi$ to every line of this derivation.
	Since $T$ is closed under modal propositional substitutions, every substituted instance of a $T$-axiom is a substitution instance of some member of $T$, and hence is true at every cone-world by assumption.
	And substituted instances of propositional tautologies and of $\axiomf{K}$ and $\axiomf{Dual}$ are true at every world by Kripke semantics.
	Soundness of modus ponens and necessitation then yields that $\theta(\bar\psi)$ is true at every world on the cone above $W_0$, in particular at $W_0$.
	Therefore $\theta$ is valid at $W_0$ (with respect to the given substitution language). As $\theta$ was arbitrary, the normal closure of $T$ is valid at $W_0$.
\end{proof}

We are now prepared to establish that for all Kripke categories, the propositional modal theory \theoryf{S4} is the common lower bound on the propositional modal validities. The following theorem demonstrates that \theoryf{S4} is always valid, at any world in any Kripke category. This lower bound is optimal in the sense that it can be realized as the exact set of validities of some world in a particular Kripke category for a specific language of substitution instances. Notably, arithmetic potentialism exhibits exactly \theoryf{S4} in many instances~\cite{Hamkins:The-modal-logic-of-arithmetic-potentialism}.

\begin{theorem}\label{Theorem.S4-always-valid}
	The propositional modal theory \theoryf{S4} is valid at any world in any Kripke category. Suppose $\KripkeCat$ is a Kripke category of $\commonL$-structures. The propositional modal theory \theoryf{S4} is valid at every world of $\KripkeCat$ with respect to any language of substitution instances, even with parameters allowed. In other words, for any such world $W_0$, we have that
	\begin{equation*}
		\theoryf{S4} \of \Val_{\KripkeCat}(W_0, \LDiamond_{W_0}).
	\end{equation*}
\end{theorem}
\begin{proof}
	Recall that $\theoryf{S4}$ is the normal closure of the axioms $\axiomf{K}$, $\axiomf{Dual}$, $\axiomf{T}$ and $\axiomf{4}$.
	Fix a Kripke category $\KripkeCat$ and an arbitrary world $W_0$.
	Let $W$ be any world on the cone above $W_0$.

	We verify that each axiom holds at $W$ for arbitrary substitution instances.

	\textit{$\axiomf{K}$.} Suppose $W\satisfies\necessary(\varphi\implies\psi)$ and $W\satisfies\necessary\varphi$.
	For any accessibility mapping $f\colon W\to U$, we have $U\satisfies\varphi\implies\psi$ and $U\satisfies\varphi$, hence $U\satisfies\psi$.
	Thus $W\satisfies\necessary\psi$.

	\textit{$\axiomf{Dual}$.} This is valid in every Kripke semantics by the definition of $\necessary$ as the dual of $\possible$.

	\textit{$\axiomf{T}$.} If $W\satisfies\necessary\varphi$, then taking the identity accessibility mapping $\id_W\colon W\to W$ yields $W\satisfies\varphi$.

	\textit{$\axiomf{4}$.} Suppose $W\satisfies\necessary\varphi$.
	Let $f\colon W\to U$ be any accessibility mapping and let $g\colon U\to V$ be any accessibility mapping.
	Then $g\compose f\colon W\to V$ is an accessibility mapping, so $V\satisfies\varphi$.
	Since $g$ was arbitrary, $U\satisfies\necessary\varphi$; and since $f$ was arbitrary, $W\satisfies\necessary\necessary\varphi$.

	Therefore every substitution instance of $\axiomf{K}$, $\axiomf{Dual}$, $\axiomf{T}$ and $\axiomf{4}$ is true at every world on the cone above $W_0$.
	Let $T$ be the set of all propositional substitution instances of these axioms. Then $T$ is closed under modal propositional substitutions and its normal closure is $\theoryf{S4}$.
	By the Cone Lemma, $\theoryf{S4}$ is valid at $W_0$ with respect to any language of substitution instances.
\end{proof}

Suppose $\KripkeCat$ is a Kripke category of $\commonL$-structures, and $W$ is a world in that category. We say that an assertion $b$ in the first-order modal language $\LDiamond$ is a \emph{weak button}\label{Definition.weak-button} at $W$ if $W \satisfies \possible \necessary b$; it is \emph{unpushed} at $W$ if $W$ does not satisfy $\necessary b$ yet. Weak buttons $b_0$ and $b_1$ are \emph{independent} at $W$ if $W \satisfies \possible (\necessary b_0 \land \neg b_1)$ and $W \satisfies \possible (\neg b_0 \land \necessary b_1)$. A weak button is a type of a \emph{control statement}, first introduced in~\cite{HamkinsLoewe2008:TheModalLogicOfForcing}. These statements will later be useful in determining upper bounds on the propositional modal validities.

\begin{lemma}\label{Lemma.S4.3-valid-iff-no-independent-buttons}
	Suppose that $W_0$ is a world in a Kripke category $\KripkeCat$ of $\commonL$-structures. The propositional modal theory \theoryf{S4.3} is valid at $W_0$ with respect to a given language of substitution instances if and only if there is no pair of independent weak buttons at $W_0$ in that language. In particular, for any such world $W_0$, we have that
	\begin{equation*}
		\theoryf{S4.3} \of \Val_{\KripkeCat}(W_0, \mathscr L),
	\end{equation*}
	where $\mathscr L$ is the relevant language of substitution instances.
\end{lemma}
\begin{proof}
	Fix a world $W$ on the cone above $W_0$ and let $\varphi$ and $\psi$ be any substitution instances in the given language.
	By propositional reasoning and duality, the negation of the instance of $\axiomf{.3}$,
\begin{equation*}
		\necessary(\necessary\varphi\implies\psi)\,\vee\,\necessary(\necessary\psi\implies\varphi),
\end{equation*}
	is equivalent to
\begin{equation*}
		\possible(\necessary\varphi\wedge\neg\psi)\ \wedge\ \possible(\necessary\psi\wedge\neg\varphi).
\end{equation*}
	Thus, $\axiomf{.3}$ fails at $W$ under some substitution instance iff there are two independent weak buttons at $W$ (namely $\varphi$ and $\psi$).

	$(\Rightarrow)$ If $\theoryf{S4.3}$ is valid at $W_0$, then in particular every instance of $\axiomf{.3}$ holds at $W_0$, and hence no pair of independent weak buttons exists at $W_0$ in the given language.

	$(\Leftarrow)$ Conversely, assume that there is no pair of independent weak buttons at $W_0$ in the given language.
	Suppose, towards a contradiction, that $\axiomf{.3}$ fails at some world $W$ on the cone above $W_0$ under a substitution instance $\varphi,\psi$.
	Then $\varphi$ and $\psi$ are independent weak buttons at $W$.
	Let $f\colon W_0\to W$ witness that $W$ lies on the cone above $W_0$.
	The witnesses for independence at $W$ compose with $f$ to give witnesses for independence at $W_0$, contradicting the assumption.
	Therefore, every substitution instance of $\axiomf{.3}$ holds at every world on the cone above $W_0$.

	Since $\theoryf{S4}$ is valid at every world in every Kripke category (Theorem \ref{Theorem.S4-always-valid}), all axioms of $\theoryf{S4}$ hold throughout that cone as well.
	By the Cone Lemma, the normal closure $\theoryf{S4.3}$ is valid at $W_0$.
\end{proof}

A direct method to obtain lower bounds on the propositional modal validities of a world involves examining particular shapes of diagrams in the cone above it. One of the most notable shapes is an \emph{upwards span}, or simply a \emph{span}, illustrated below. A morphism constitutes the simplest shape of all, and can be viewed as a \emph{degenerate span}.
$$
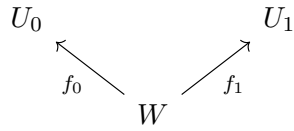
\begin{tikzcd}
		U_0 &                                        & U_1 \\
		& W \arrow[lu, "f_0"] \arrow[ru, "f_1"'] &
	\end{tikzcd}$$

Suppose $A$ is a (possibly empty) set of parameters from a world $W$. We say that a span $U_0 \xleftarrow{f_0} W \xrightarrow{f_1} U_1$ is \emph{amalgamable} for parameters from $A$, or just \emph{$A$-amalgamable}, if it can be completed in the category by a downwards span $U_0 \xrightarrow{g_0} U \xleftarrow{g_1} U_1$ to form a square diagram that commutes over $A$, meaning that $(g_0 \compose f_0)(a) = (g_1 \compose f_1)(a)$ for any $a \in A$. Likewise, the span is \emph{factorable} for parameters from $A$, or \emph{$A$-factorable}, if it can be completed by either $g_0 \colon U_0 \to U_1$ or $g_1 \colon U_1 \to U_0$ in the category to form a triangular diagram that commutes for any $a \in A$. A degenerate span $W \xrightarrow{f} U$ is \emph{invertible} for parameters from $A$, or just \emph{$A$-invertible}, if it can be completed by some $g \colon U \to W$ in the category so that $(g \compose f){\restrict_A} = \id_A$. Each situation is illustrated below; the symbol $\circlearrowleft_A$ in the center of a diagram indicates that the diagram commutes over $A$. If $A$ is empty, then the commutativity requirement is vacuous. For example, an $\emptyset$-amalgamable span is the same as a \emph{convergent}, or a \emph{directed}, span.

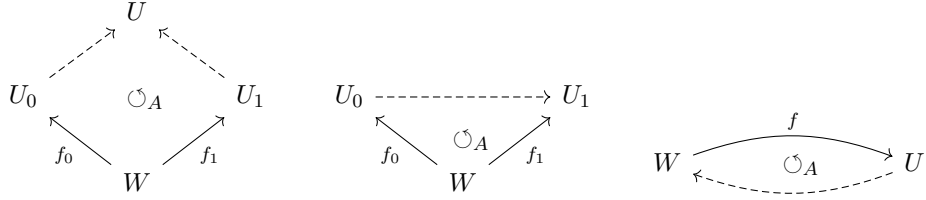
\begin{figure}[H]
	\tikzset{column sep=small, ampersand replacement=\&}
	\begin{floatrow}[3]
		\centering
		\ffigbox{\adjustbox{scale=.9}{\begin{tikzcd}
				\& U    \arrow[dd, phantom, "\prescript{\phantom{A}}{}{\circlearrowleft}_{A}", midway] \& \\
				U_0 \arrow[ru, dashed] \& \& U_1 \arrow[lu, dashed] \\
				\& W \arrow[lu, "f_0"] \arrow[ru, "f_1"'] \&
			\end{tikzcd}}}{}
		\ffigbox{\adjustbox{scale=.9}{\begin{tikzcd}
				U_0 \arrow[rr, dashed] \& \arrow[d, phantom, "\prescript{\phantom{A}}{}{\circlearrowleft}_{A}", midway] \& U_1 \\
				\& W \arrow[lu, "f_0"] \arrow[ru, "f_1"'] \&
			\end{tikzcd}}}{}
		\ffigbox{\adjustbox{scale=.9}{\begin{tikzcd}
				W \arrow[rr, bend left=20, "f"] \& \prescript{\phantom{A}}{}{\circlearrowleft}_{A} \& U \arrow[ll, bend left=20, dashed]
			\end{tikzcd}}}{}
	\end{floatrow}
\end{figure}

Consider the cone above a world $W_0$. If, for every accessibility mapping $f \colon W_0 \to W$, every span $U_0 \leftarrow W \rightarrow U_1$ is $f(A)$-amalgamable, then we say the cone \emph{$A$-amalgamable}. If every such span is $f(A)$-factorable, then the cone is called \emph{$A$-factorable}. Similarly, if every degenerate span $W \rightarrow U$ is $f(A)$-invertible, the cone is \emph{$A$-invertible}. For any such $A$-property associated with the cone above $W_0$, if $A$ is equal to $W_0$, we usually omit $A$-prefix. Consequently, for example, we write ``factorable'' instead of ``$W_0$-factorable.''

\newtheorem*{Lowerboundstheorem}{First Lower Bounds Theorem}%
\begin{Lowerboundstheorem}\makeatletter\def\@currentlabelname{first lower bounds theorem}\makeatother\label{Lowerboundstheorem}
	Suppose that $W_0$ is a world in a Kripke category $\KripkeCat$ of $\commonL$-structures and $A \of W_0$ is a set of parameters.
	\begin{enumerate}
		\item If the cone above $W_0$ is $A$-amalgamable, then the propositional modal theory \theoryf{S4.2} is valid at $W_0$ for substitution instances in the language $\LDiamond$ with parameters from $A$.
		\item If the cone above $W_0$ is $A$-factorable, then the propositional modal theory \theoryf{S4.3} is valid at $W_0$ for all substitution instances in the language $\LDiamond$ with parameters from $A$.
		\item If the cone above $W_0$ is $A$-invertible, then the propositional modal theory \theoryf{S5} is valid at $W_0$ for all substitution instances in the language $\LDiamond$ with parameters from $A$.
	\end{enumerate}
\end{Lowerboundstheorem}
\begin{proof}
	Let $W$ be a world on the cone above $W_0$ and let $f\colon W_0\to W$ witness this.
	Since \theoryf{S4} is always valid in any Kripke category (Theorem \ref{Theorem.S4-always-valid}), $\theoryf{S4}$ is valid at $W$.
	Therefore, by \nameref{ConeLemma}, it is enough to show that $\axiomf{.2}$, $\axiomf{.3}$ and $\axiomf{5}$ hold at $W$ for all substitution instances with parameters from $f(A)$.

	\smallskip
	\noindent\textit{(1) $A$-amalgamable $\Rightarrow$ $\axiomf{.2}$ at $W$.}
	Suppose $W\satisfies\possible\necessary\varphi[f(\bar a)]$ where $\bar a$ is a tuple of parameters from $A$.
	Then there is an accessibility mapping $f_0\colon W\to U_0$ such that
\begin{equation*}
		U_0\satisfies\necessary\varphi[(f_0\compose f)(\bar a)].
\end{equation*}
	Let $f_1\colon W\to U_1$ be arbitrary.
	Since the cone above $W_0$ is $A$-amalgamable, the span $U_0\xleftarrow{\,f_0\,}W\xrightarrow{\,f_1\,}U_1$ is $f(A)$-amalgamable, so there are accessibility mappings $g_0\colon U_0\to U$ and $g_1\colon U_1\to U$ such that
\begin{equation*}
		\bigl(g_0\compose f_0\compose f\bigr)\restrict A = \bigl(g_1\compose f_1\compose f\bigr)\restrict A.
\end{equation*}
	From $U_0\satisfies\necessary\varphi[(f_0\compose f)(\bar a)]$ and the accessibility mapping $g_0$ we get
\begin{equation*}
		U\satisfies\varphi[(g_0\compose f_0\compose f)(\bar a)].
\end{equation*}
	By commutativity on $A$, the parameter tuples agree, so also $U\satisfies\varphi[(g_1\compose f_1\compose f)(\bar a)]$.
	Thus $g_1$ witnesses
\begin{equation*}
		U_1\satisfies\possible\varphi[(f_1\compose f)(\bar a)].
\end{equation*}
	Since $f_1$ was arbitrary, we have $W\satisfies\necessary\possible\varphi[f(\bar a)]$, that is the instance of $\axiomf{.2}$ holds at $W$.

	\smallskip
	\noindent\textit{(2) $A$-factorable $\Rightarrow$ $\axiomf{.3}$ at $W$.}
	Suppose towards a contradiction that there are two independent weak buttons $b_0[f(\bar a)]$ and $b_1[f(\bar a)]$ at $W$, with $\bar a$ a tuple of parameters from $A$.
	Then there exist accessibility mappings $f_0\colon W\to U_0$ and $f_1\colon W\to U_1$ such that
\begin{align*}
		U_0&\satisfies\necessary b_0[(f_0\compose f)(\bar a)]\wedge \neg b_1[(f_0\compose f)(\bar a)],\\
		U_1&\satisfies\necessary b_1[(f_1\compose f)(\bar a)]\wedge \neg b_0[(f_1\compose f)(\bar a)].
\end{align*}
	Because the cone above $W_0$ is $A$-factorable, the span $U_0\xleftarrow{\,f_0\,}W\xrightarrow{\,f_1\,}U_1$ is $f(A)$-factorable.
	Hence either there is an accessibility mapping $g_0\colon U_0\to U_1$ with $\bigl(g_0\compose f_0\compose f\bigr)\restrict A = \bigl(f_1\compose f\bigr)\restrict A$, or there is an accessibility mapping $g_1\colon U_1\to U_0$ with $\bigl(g_1\compose f_1\compose f\bigr)\restrict A = \bigl(f_0\compose f\bigr)\restrict A$.
	In the first case, from $U_0\satisfies\necessary b_0[(f_0\compose f)(\bar a)]$ we obtain
\begin{equation*}
		U_1\satisfies b_0[(g_0\compose f_0\compose f)(\bar a)],
\end{equation*}
	which by commutativity contradicts $U_1\satisfies\neg b_0[(f_1\compose f)(\bar a)]$.
	The second case is symmetric.
	Therefore no pair of independent weak buttons exists at $W$, and by Lemma \ref{Lemma.S4.3-valid-iff-no-independent-buttons} the axiom $\axiomf{.3}$ holds at $W$ for every substitution instance with parameters from $f(A)$.

	\smallskip
	\noindent\textit{(3) $A$-invertible $\Rightarrow$ $\axiomf{5}$ at $W$.}
	Suppose $W\satisfies\possible\necessary\varphi[f(\bar a)]$ where $\bar a$ is a tuple of parameters from $A$.
	Then there exists an accessibility mapping $g\colon W\to U$ such that $U\satisfies\necessary\varphi[(g\compose f)(\bar a)]$.
	Since the cone above $W_0$ is $A$-invertible, the morphism $g$ has a pseudo-inverse $h\colon U\to W$ with $(h\compose g)(x)=x$ for all $x\in f(A)$.
	In particular, $(h\compose g\compose f)(\bar a)=f(\bar a)$.
	As $U\satisfies\necessary\varphi[(g\compose f)(\bar a)]$ and $h$ is an accessibility mapping, we conclude
\begin{equation*}
		W\satisfies\varphi[(h\compose g\compose f)(\bar a)]\;=\;\varphi[f(\bar a)].
\end{equation*}
	Thus the instance of $\axiomf{5}$ holds at $W$.
\end{proof}

A collection of statements forms a \emph{labeling}\label{Definition.Labeling} of a propositional Kripke frame if there exists an assignment of these statements to the frame such that the modalities precisely align with the frame order. More specifically, a labeling of a propositional Kripke frame $F$ for a world $W_0$ in a Kripke category is an assignment to each node $w \in F$ a statement $\Phi_w$ from a given language. This assignment has to fulfill the following conditions. First, $W_0$ needs to satisfy $\Phi_{w_0}$, where $w_0$ is a designated initial element of $F$. Next, every world in the cone above $W_0$ satisfies exactly one $\Phi_w$. And finally, the relation $w \leq_F v$ holds if and only if each world in the cone satisfying $\Phi_w$ can access a world that satisfies $\Phi_v$. The following lemma is just a minor generalization of~\cite[lemma 9]{HamkinsLeibmanLoewe2015:StructuralConnectionsForcingClassAndItsModalLogic}.

\newtheorem*{LabelingLemma}{Labeling Lemma}%
\begin{LabelingLemma}\makeatletter\def\@currentlabelname{labeling lemma}\makeatother\label{LabelingLemma}
	Suppose that $w \mapsto \Phi_w$ is a labeling of a finite frame $F$ for a world $W_0$ in a Kripke category $\KripkeCat$ of $\commonL$-structures and that $w_0$ is an initial world of $F$. Then, for any propositional Kripke model whose frame is $F$, there is an assignment $p \mapsto \psi_p$ of the propositional variables to assertions in the Boolean closure of the language of the labeling statements such that for any propositional modal assertion $\varphi$, we have
	\begin{equation*}
		(M,w_0) \satisfies \varphi(p_0,\ldots,p_n) \qquad \text{if and only if} \qquad W_0 \satisfies \varphi(\psi_{p_0},\ldots,\psi_{p_n}).
	\end{equation*}
\end{LabelingLemma}

The~\nameref{LabelingLemma} is the main tool for translating finite Kripke-frame arguments into statements about a given world. In practice we use it through \emph{control statements}: finite families of assertions whose modal behavior is sufficiently uniform to label finite frames.

The control statements used below are standard. A \emph{button} is an assertion that is necessarily possibly necessary; it is \emph{pushed} if already necessary and \emph{unpushed} otherwise. A \emph{switch} is an assertion whose truth value can necessarily be toggled on and off. A \emph{dial} is a list $d=\<d_i\mid i<n>$ such that necessarily exactly one $d_i$ holds, but each $d_i$ is possible. A \emph{ratchet} is a list $\<r_i\mid 1\le i\le n>$ whose later entries necessarily imply the earlier ones and whose volume can be increased one step at a time; by convention, the empty sequence is an uncranked ratchet of length $0$.

We call a button $b$ \emph{pure} if
\begin{equation*}
    \necessary(b \rightarrow \necessary b).
\end{equation*}
Equivalently (since $\axiomf{T}$ holds in every Kripke category), $b\leftrightarrow \necessary b$; in particular, for a pure button, being unpushed is the same as being false.

\begin{observation}\label{Observation.Under-S4.2-weak-buttons-are-buttons}
	Suppose that $\KripkeCat$ is a Kripke category of $\commonL$-structures, and $W_0$ is a world in $\KripkeCat$ at which the propositional modal theory $\theoryf{S4.2}$ is valid with respect to a language of substitution instances. Then, at $W_0$, every weak button expressible in the language of substitutions is a button.
\end{observation}

\begin{proof}
	Suppose that $W_0 \satisfies \possible \necessary \varphi$ for an assertion $\varphi$ in the language of substitutions, so that $\varphi$ is a weak button. Since $\theoryf{S4}$ is valid, we obtain $W_0 \satisfies \possible \necessary \necessary \varphi$. Furthermore, due to the validity of $\theoryf{S4.2}$, we have $W_0 \satisfies \necessary \possible \necessary \varphi$. Consequently, $\varphi$ is a button at $W_0$.
\end{proof}

A \emph{railyard} is a labeling of a finite pre-tree frame. A \emph{tree} is a partially ordered set $\<T,\leq>$ such that for each $t\in T$ the set $\{u\mid u\leq t\}$ is linearly ordered. A \emph{pretree} is obtained by replacing nodes of a tree by clusters of equivalent nodes.

A control statement is \emph{normal} if its controlling condition is already necessary at the world of evaluation; buttons, switches, dials, and fully cranked ratchets are typical examples. The following upper-bound theorem collects the standard frame-labeling consequences we shall use later; it is adapted from the potentialist literature cited in~\cite{HamkinsLinnebo:Modal-logic-of-set-theoretic-potentialism,Hamkins:The-modal-logic-of-arithmetic-potentialism,HamkinsWoloszyn:Modal-model-theory}.

\newtheorem*{Upperboundstheorem}{First Upper Bounds Theorem}%
\begin{Upperboundstheorem}\makeatletter\def\@currentlabelname{first upper bounds theorem}\makeatother\label{Upperboundstheorem}
	Suppose $W_0$ is a world in a Kripke category of $\commonL$-structures.
	\nobreak
	\begin{enumerate}
		\item The propositional modal validities of $W_0$ are always contained in the trivial propositional modal theory \theoryf{Triv}.
		\item If $W_0$ has arbitrarily large finite independent families of switches or dials, then the propositional modal validities of $W_0$ are contained in the propositional modal theory \theoryf{S5}.
		\item If $W_0$ has arbitrarily large finite independent families of switches or dials together with arbitrarily long finite uncranked ratchets, and these are independent, then the propositional modal validities of $W_0$ are contained in the propositional modal theory \theoryf{S4.3}.
		\item If $W_0$ has arbitrarily large finite independent families of switches or dials together with arbitrarily large finite independent families of unpushed buttons, and these are independent, then the propositional modal validities of $W_0$ are contained in the propositional modal theory \theoryf{S4.2}.
		\item If $W_0$ admits a railyard labeling for every finite pretree $T$---and it suffices to handle only pretrees with all clusters the same size and all branching clusters having the same degree---then the propositional modal validities of $W_0$ are contained in the propositional modal theory \theoryf{S4}.
	\end{enumerate}
	In each case, the relevant language of substitution instances would be any nonempty language closed under Boolean connectives and containing the control statements.
\end{Upperboundstheorem}

\subsection{Grzegorczyk logic}
The propositional modal theory \theoryf{Grz} will be used repeatedly below. Following~\cite{Woloszyn:Grzegorczyk}, call $\varphi$ \emph{penultimate} if it is true, possibly false, and once false remains necessarily false; equivalently,
\begin{equation*}
	\varphi \land \possible\neg\varphi \land \necessary(\neg\varphi \implies \necessary\neg\varphi).
\end{equation*}
An assertion \emph{attains} a penultimate truth-value if either it or its negation is penultimate. The characterization from~\cite{Woloszyn:Grzegorczyk} yields the following lower-bound theorem.

\newtheorem*{Lowerboundstheorem2}{Second Lower Bounds Theorem}%
\begin{Lowerboundstheorem2}\makeatletter\def\@currentlabelname{second lower bounds theorem}\makeatother\label{Lowerboundstheorem2}
	Suppose that $W_0$ is a world in a Kripke category $\KripkeCat$ of $\commonL$-structures, and that a propositional modal theory $\Lambda \fo \theoryf{S4}$ is valid at $W_0$ for a given language of substitution instances. If for every world $W$ accessible from $W_0$ and every assertion $\varphi$ contingent at $W$ possibly attains a penultimate truth-value, then the propositional modal theory $\Lambda+\theoryf{Grz}$ is valid at $W_0$ with respect to the language of substitution instances.
\end{Lowerboundstheorem2}

The corresponding upper bounds are the direct Kripke-category analogues of results from~\cite{Woloszyn:Grzegorczyk}.

\newtheorem*{Upperboundstheorem2}{Second Upper Bounds Theorem}%
\begin{Upperboundstheorem2}\makeatletter\def\@currentlabelname{second upper bounds theorem}\makeatother\label{Upperboundstheorem2}
	Suppose $W_0$ is a world in a Kripke category of $\commonL$-structures.
	\nobreak
	\begin{enumerate}
		\item If $W_0$ has arbitrarily long finite uncranked ratchets, then the propositional modal validities of $W_0$ are contained in the propositional modal theory \theoryf{Grz.3}.
		\item If $W_0$ has arbitrarily large finite independent families of unpushed buttons, then the propositional modal validities of $W_0$ are contained in the propositional modal theory \theoryf{Grz.2}.
		\item If $W_0$ has a railyard labeling for every finite tree $T$---and it suffices only to handle trees where all branching nodes have the same degree---then, the propositional modal validities are contained in the propositional modal theory \theoryf{Grz}.
	\end{enumerate}
	In each case, the relevant language of substitution instances would be any language closed under Boolean connectives and containing the respective control statements.
\end{Upperboundstheorem2}

\section{Modality elimination in the category of sets}\label{Section.Modality-elimination-in-Sets}

In this section we prove the~\nameref{Modalityeliminationtheorem.Sets} for the category of sets, with respect to all classes of morphisms specified in the~\nameref{Maintheorem.Sets}. This theorem serves as the backbone of the major results established later.

\newtheorem*{Modalityeliminationtheorem.Sets}{Modality-and-Quantifier Elimination Theorem}%
\begin{Modalityeliminationtheorem.Sets}\makeatletter\def\@currentlabelname{modality-and-quantifier elimination theorem}\makeatother\label{Modalityeliminationtheorem.Sets}
The categories of sets and finite sets with either functions, surjections, injections, inclusions, bijections, or identities all admit simultaneous modality and quantifier elimination down to the class of finite Boolean combinations of atomic formulas and sentences asserting that there are precisely $n$ elements. In other words, every formula in the first-order modal language of equality $\LDiamondofeq$ is equivalent to such a Boolean combination in each of these categories.
\end{Modalityeliminationtheorem.Sets}

A partition $P$ of a finite set $X$ can be seen as a formula in the language of equality $\Lofeq$ that describes a satisfiable pattern of equalities and inequalities for $|X|$-many variables. More precisely, by denoting the block of $P$ to which $j$ belongs as $[j]$, the desired formula is given by
\begin{equation*}
	\bigwedge_{\substack{i<j\\ [i]=[j]}} (x_i = x_j)
	\ \land\
	\bigwedge_{\substack{i<j\\ [i]\neq [j]}} (x_i \neq x_j).
\end{equation*}
The illustration below demonstrates the correspondence between partitions and satisfiable equality/inequality patterns, using several examples of a partition of four. The black dots represent distinct variables. We identify a pair of distinct variables if they lie in the same monochromatic region. This illustration is inspired by the \TeX{} Stack Exchange post \emph{``Set Partitions and tikz''}\footnote{\url{https://tex.stackexchange.com/questions/578823/set-partitions-and-tikz}}.
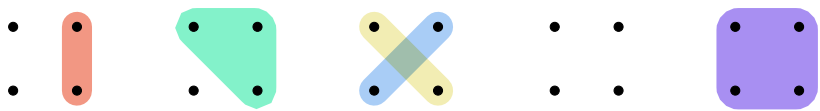
\begin{figure}[H]
$$
	\begin{tikzpicture}[scale=.6,every node/.style={transform shape},
			Line/.style = {
					line width=4mm,
					line cap=round
				},
			Area/.style = {
					line width=5mm,
					fill,
					line cap=round,
					rounded corners= 0.5mm
				}
		]
		\begin{scope}
			\foreach \k in {1,...,4}{
					\draw[fill](90*\k+225:1cm) circle (1mm) coordinate (N\k);
				}
			\begin{scope}[on background layer]
				\draw[Line,SalmonSunrise](N1) -- (N2);
			\end{scope}
		\end{scope}
		\begin{scope}[shift={(4,0)}]
			\foreach \k in {1,...,4}{
					\draw[fill](90*\k+225:1cm) circle (1mm) coordinate (N\k);
				}
			\begin{scope}[on background layer]
				\draw[Area,MintBreeze] (N1) -- (N2) -- (N3) -- cycle;
			\end{scope}
		\end{scope}
		\begin{scope}[shift={(8,0)}]
			\foreach \k in {1,...,4}{
					\draw[fill](90*\k+225:1cm) circle (1mm) coordinate (N\k);
				}
			\begin{scope}[on background layer]
				\begin{scope}[blend group=multiply]\draw[Line,LemonChiffon](N1) -- (N3); \draw[Line,SkyBlueDream](N2) -- (N4); \end{scope}
			\end{scope}
		\end{scope}
		\begin{scope}[shift={(12,0)}]
			\foreach \k in {1,...,4}{
					\draw[fill](90*\k+225:1cm) circle (1mm) coordinate (N\k);
				}
		\end{scope}
		\begin{scope}[shift={(16,0)}]
			\foreach \k in {1,...,4}{
					\draw[fill](90*\k+225:1cm) circle (1mm) coordinate (N\k);
				}
			\begin{scope}[on background layer]
				\draw[Area,LavenderHaze](N1) -- (N2) -- (N3) -- (N4) -- cycle;
			\end{scope}
		\end{scope}
	\end{tikzpicture}$$
	\refstepcounter{figure}
\end{figure}

Accessibility mappings in the Kripke categories of sets with injective functions preserve all atomic and negated atomic formulas. Thus, in each of $\Sets[$\into$]$, $\FinSets[$\into$]$, and $\InfSets[$\into$]$, the formula $\possible P(\bar x)$ is equivalent precisely to $P(\bar x)$ for each partition $P$. Naturally, the same observation holds for the corresponding categories of inclusions, bijections, and identities\label{Paragraph.Sensible-Inj}. For the remaining cases, let us order the set of all such partitions of a finite set by the \emph{refinement} relation. More precisely, given partitions $P$ and $Q$, we write $P \leq Q$ if every block of $P$ is a subset of some block of $Q$:
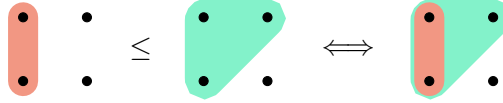
\begin{figure}[H]
$$\begin{tikzpicture}[scale=.6,every node/.style={transform shape},
			Line/.style = {
					line width=4mm,
					line cap=round
				},
			Area/.style = {
					line width=5mm,
					fill,
					line cap=round,
					rounded corners= 0.5mm
				}
		]
		\begin{scope}
			\foreach \k in {1,...,4}{
					\draw[fill](90*\k+225:1cm) circle (1mm) coordinate (N\k);
				}
			\begin{scope}[on background layer]
				\draw[Line,SalmonSunrise](N3) -- (N4);
			\end{scope}
		\end{scope}
		\begin{scope}[shift={(4,0)}]
			\foreach \k in {1,...,4}{
					\draw[fill](90*\k+225:1cm) circle (1mm) coordinate (N\k);
				}
			\begin{scope}[on background layer]
				\draw[Area,MintBreeze] (N2) -- (N3) -- (N4) -- cycle;
			\end{scope}
		\end{scope}
		\begin{scope}[shift={(9,0)}]
			\foreach \k in {1,...,4}{
					\draw[fill](90*\k+225:1cm) circle (1mm) coordinate (N\k);
				}
			\begin{scope}[on background layer]
				\draw[Area,MintBreeze] (N2) -- (N3) -- (N4) -- cycle;
				\draw[Line,SalmonSunrise] (N3) -- (N4);
			\end{scope}
		\end{scope}
		\begin{scope}[shift={(1.9,0)},scale=1.66666666667]
			\node[] at (0,0) {${\leq}$};
		\end{scope}
		\begin{scope}[shift={(6.5,0)},scale=1.66666666667]
			\node[] at (0,0) {$\Iff$};
		\end{scope}
	\end{tikzpicture}$$
\refstepcounter{figure}
\label{Figure.Refinement}
\end{figure}

It is not difficult to observe that in the categories of sets and functions and of sets and surjective functions---and likewise in their full subcategories of finite sets and of infinite sets---the relation $P \leq Q$ holds if and only if $P(\bar x)$ entails $\possible Q(\bar x)$. The only point requiring comment in the infinite subcategories is that identifying only finitely many parameter-values in an infinite set still yields an infinite quotient. On the other hand, suppose $W \models \possible Q[\bar x]$. So there exists a partition $P \leq Q$ such that $W \models P[\bar x]$. Since there can be only finitely many partitions refining $Q$, there are only finitely many corresponding $P(\bar x)$ that hold given $\possible Q(\bar x)$. Consequently,
\begin{equation*}
	\Sets \models \forall \bar x \big(\possible Q(\bar x) \iff \bigvee_{P \leq Q} P(\bar x) \big) \quad \text{and} \quad \Sets[$\onto$] \models \forall \bar x \big(\possible Q(\bar x) \iff \bigvee_{P \leq Q} P(\bar x) \big),
\end{equation*}
with the same equivalences holding in the corresponding finite and infinite full subcategories.

Every quantifier-free formula $\varphi(\bar x)$ in the non-modal language of equality $\Lofeq$ is either inconsistent or equivalent to a disjunction of finitely many partitions $P(\bar x)$~\cite[lemma 1.5.6]{ChangKeisler1990:ModelTheory}. Hence, inductively, we arrive at the following.

\begin{lemma}\label{Lemma.QF-QE-ME}
	Suppose $\varphi(\bar x)$ is a satisfiable quantifier-free formula in the first-order modal language of equality $\LDiamondofeq$. The formula $\varphi(\bar x)$ is equivalent to a finite disjunction of partitions of the free variables of $\bar x$, in any of the categories of sets, finite sets, or infinite sets with either functions, surjections, injections, inclusions, bijections, or identities.
\end{lemma}

It is not hard to show that every formula $\varphi(\bar x)$ in the non-modal language of equality $\Lofeq$ is equivalent to a finite Boolean combination of partitions of the free-variables of $\bar x$ and sentences of the form $\sigma_k$ asserting that there are precisely $k$ elements (see~\cite[theorem 1.5.7]{ChangKeisler1990:ModelTheory}). Equivalently, every satisfiable formula $\varphi(\bar x)$ in the language $\Lofeq$ is equivalent to a finite disjunction of assertions of the form $Q(\bar x) \land \sigma_{\geq k}$ and $Q(\bar x) \land \sigma_{\leq k}$, where $\sigma_{\geq k}$ asserts that there are at least $k$ elements and $\sigma_{\leq k}$ asserts that there are at most $k$ elements. One can therefore eliminate quantifiers down to the class of such disjunctions. It remains to show that modalities can also be eliminated from formulas of the form
\begin{equation*}
\possible\bigl(Q(\bar x)\land\sigma_k\bigr),
\qquad
\possible\bigl(Q(\bar x)\land\sigma_{\leq k}\bigr),
\qquad\text{and}\qquad
\possible\bigl(Q(\bar x)\land\sigma_{\geq k}\bigr),
\end{equation*}
where $Q(\bar x)$ is a partition formula. The ``at most'' case reduces to the exact-cardinality case, since $\sigma_{\leq k}$ is the finite disjunction $\bigvee_{j\leq k}\sigma_j$. We handle the exact-cardinality and ``at least'' cases directly. Notice also that $\sigma_{\geq r}$ is itself a finite Boolean combination of exact-cardinality sentences, since $\sigma_{\geq r}$ is equivalent to $\neg\bigvee_{j<r}\sigma_j$.

\medskip

We first dispose of the degenerate exact-cardinality cases. If $k<|Q|$, then $Q(\bar x)\land\sigma_k$ is inconsistent. In particular, when $k=0$ this covers every nonempty tuple of variables. The remaining zero-cardinality case is the unique partition of the empty tuple, in which case $Q(\bar x)\land\sigma_0$ is simply $\sigma_0$; and in all the categories of sets under consideration, $\possible\sigma_0$ is equivalent to $\sigma_0$, since the empty set is accessible precisely from the empty set. Thus, in the exact-cardinality case, we may assume below that $k>0$ and $k\geq |Q|$.

\medskip

In the categories $\Sets$ and $\FinSets$, an arbitrary function can realize any coarsening of the current equality pattern of the tuple, and the target set may be chosen to have any prescribed positive size large enough for that pattern. Therefore, for $k>0$ with $k\geq |Q|$,
\begin{equation*}
\possible\bigl(Q(\bar x)\land\sigma_k\bigr)
\quad\text{is equivalent to}\quad
\bigvee_{P\leq Q}P(\bar x),
\end{equation*}
and similarly
\begin{equation*}
\possible\bigl(Q(\bar x)\land\sigma_{\geq k}\bigr)
\quad\text{is equivalent to}\quad
\bigvee_{P\leq Q}P(\bar x).
\end{equation*}
Thus the required elimination follows in these cases from lemma~\ref{Lemma.QF-QE-ME} and Boolean closure.

\medskip

The case of surjections is more subtle, since a surjective image cannot have larger cardinality than its source. Recall that $\possible Q(\bar x)$ is equivalent in $\Sets[$\onto$]$ or $\FinSets[$\onto$]$ to a finite disjunction of its refinements $P(\bar x)$. For exact cardinality, with $k>0$ and $k\geq |Q|$, we have
\begin{equation*}
\possible\bigl(Q(\bar x)\land\sigma_k\bigr)
\quad\text{equivalent to}\quad
\bigvee_{P\leq Q}
\Bigl(P(\bar x)\land\sigma_{\geq k+|P|-|Q|}\Bigr).
\end{equation*}
Indeed, if the current tuple realizes $P$, then its $|P|$ distinct parameter-values must map onto the $|Q|$ parameter-values in the target pattern $Q$, and the remaining $k-|Q|$ target elements require at least that many further preimages. Conversely, whenever the displayed lower bound on the source size holds, we may choose a surjection onto a $k$-element target realizing $Q$. Likewise, for the ``at least'' case,
\begin{equation*}
\possible\bigl(Q(\bar x)\land\sigma_{\geq k}\bigr)
\quad\text{is equivalent to}\quad
\bigvee_{P\leq Q}
\Bigl(P(\bar x)\land\sigma_{\geq |P|+\max(0,k-|Q|)}\Bigr).
\end{equation*}
The right-hand sides are finite Boolean combinations of partition formulas and exact-cardinality sentences. Hence the result follows for $\Sets[$\onto$]$ and $\FinSets[$\onto$]$.

\medskip

For inclusions, the equality pattern of the tuple is preserved, while the ambient set can only grow. Hence
\begin{equation*}
\possible\bigl(Q(\bar x)\land\sigma_{\geq k}\bigr)
\quad\text{is equivalent to}\quad
Q(\bar x),
\end{equation*}
and
\begin{equation*}
\possible\bigl(Q(\bar x)\land\sigma_j\bigr)
\quad\text{is equivalent to}\quad
Q(\bar x)\land\sigma_{\leq j}
\end{equation*}
for exact cardinalities. The ``at most'' case then follows by finite disjunction. This proves the cases of $\Sets[$\incl$]$ and $\FinSets[$\incl$]$.

\medskip

The remaining cases follow by the \nameref{Renaming-lemma} and lemma~\ref{Lemma.QF-QE-ME} combined. The proof of the~\nameref{Modalityeliminationtheorem.Sets} is now completed.

For the corresponding full subcategories of infinite sets, the elimination result simplifies further, because every exact-cardinality sentence $\sigma_n$ is false there. We record the resulting partition form explicitly.

\begin{corollary}
\label{Corollary.InfSets-formulas-partitions}
    Let $\mathcal C$ be any of the categories $\InfSets$, $\InfSets[$\onto$]$, $\InfSets[$\into$]$, $\InfSets[$\incl$]$, $\InfSets[$\bij$]$, or $\InfSets[$\ident$]$. Then every formula $\varphi(\bar x)$ in $\LDiamondofeq$ is equivalent in $\mathcal C$ to a finite disjunction of partition formulas on the free variables of $\bar x$. Equivalently, for every world $W$ in $\mathcal C$, every accessibility mapping $f\colon W\to U$, and every tuple $\bar a\in W^{|\bar x|}$, the truth of $\varphi[f(\bar a)]$ at $U$ depends only on the partition of $\{0,\dots,|\bar x|-1\}$ induced by equality among the entries of $f(\bar a)$.
\end{corollary}

\begin{proof}
    Let $\Lambda$ be the class of finite Boolean combinations of partition formulas. By~\cite[theorem 1.5.7]{ChangKeisler1990:ModelTheory}, every non-modal formula in $\Lofeq$ is equivalent to a finite Boolean combination of partition formulas together with exact-cardinality sentences $\sigma_n$. Since every object of $\mathcal C$ is infinite, each $\sigma_n$ is false at every world of $\mathcal C$, and hence every non-modal formula is equivalent in $\mathcal C$ to a formula in $\Lambda$.

    We now show by induction on formula complexity that every formula in $\LDiamondofeq$ is equivalent in $\mathcal C$ to a formula in $\Lambda$. The atomic case is immediate, and Boolean combinations preserve membership in $\Lambda$. For the quantifier step, if $\psi(\bar x,y)$ is equivalent in $\mathcal C$ to some $\theta(\bar x,y)\in\Lambda$, then $\exists y\,\psi(\bar x,y)$ is equivalent to the non-modal $\Lofeq$-formula $\exists y\,\theta(\bar x,y)$, which by the first paragraph is again equivalent in $\mathcal C$ to a formula in $\Lambda$. Finally, for the modal step, if $\psi(\bar x)$ is equivalent in $\mathcal C$ to some $\theta(\bar x)\in\Lambda$, then $\possible\psi(\bar x)$ is equivalent to $\possible\theta(\bar x)$. If $\theta(\bar x)$ is inconsistent, then $\possible\theta(\bar x)$ is equivalent to false. Otherwise $\possible\theta(\bar x)$ is a satisfiable quantifier-free formula in $\LDiamondofeq$, and lemma~\ref{Lemma.QF-QE-ME} implies that it is equivalent in $\mathcal C$ to a finite disjunction of partition formulas, hence to a formula in $\Lambda$. This completes the induction.

    Every tuple realizes exactly one partition, so every Boolean combination of partition formulas is equivalent to the disjunction of those partitions on which it is true. This yields the displayed normal form. The final claim follows immediately.
\end{proof}

\begin{corollary}
\label{Corollary.InfSets-sentences-Triv}
    Let $\mathcal C$ be any of the categories $\InfSets$, $\InfSets[$\onto$]$, $\InfSets[$\into$]$, $\InfSets[$\incl$]$, $\InfSets[$\bij$]$, or $\InfSets[$\ident$]$. Then every sentence of $\LDiamondofeq$ is equivalent in $\mathcal C$ to either true or false. Consequently, for every world $W$ in $\mathcal C$ and every first-order language $\mathscr L$ such that $\Lofeq\subseteq\mathscr L\subseteq\LDiamondofeq$, we have
	\begin{equation*}
		\Val_{\mathcal C}(W,\mathscr L)=\theoryf{Triv}.
	\end{equation*}
\end{corollary}

\begin{proof}
    Apply corollary~\ref{Corollary.InfSets-formulas-partitions} to formulas with no free variables. There is only one partition of the empty set, so every sentence in $\LDiamondofeq$ is equivalent in $\mathcal C$ either to that unique partition formula, namely true, or to false. It follows that a propositional modal assertion is valid at $W$ under sentential substitutions exactly when it is valid under all truth-value substitutions, that is, exactly when it belongs to \theoryf{Triv}.
\end{proof}

\section{The modal theory of the category of sets and functions}\label{Section.category-of-sets-with-functions}

We now turn to the modal theory of sets and functions, considering the category $\Sets$ of sets and functions together with its full subcategories $\FinSets$ of finite sets and $\InfSets$ of infinite sets. Note that these categories deviate from the standard definition of a potentialist system. They admit morphisms that merge individuals, which are accessibility mappings that fail to be inflationary in the domains of structures.

\begin{observation}\label{Observation.Sets-and-FinSets-same}
	A finite world in the category of sets and functions exhibits the exact same modal truths as it would when considered in the category of all finite sets and functions. More precisely, if $W$ is a finite world, then for every $\LDiamondofeq$-assertion $\varphi(\bar x)$ (with parameters allowed) and every tuple $\bar a\in W$,
	\begin{equation*}
		W \models_{\Sets} \varphi[\bar a]
		\qquad \text{if and only if} \qquad
		W \models_{\FinSets} \varphi[\bar a].
	\end{equation*}
	Consequently, any world in the category $\FinSets$ admits the same propositional modal validities as it would in the category $\Sets$.
\end{observation}

\begin{proof}
	Let $W$ be finite and let $\varphi(\bar x)$ be an assertion in $\LDiamondofeq$, with parameters allowed. Inspecting the function case of the~\nameref{Modalityeliminationtheorem.Sets}, the elimination of modal operators is identical in $\Sets$ and in $\FinSets$: exact positive targets and ``at least'' targets reduce to the same finite disjunctions of partition formulas, ``at most'' targets are finite disjunctions of exact targets, and the empty target is handled in both categories by the equivalence $\possible\sigma_0\leftrightarrow \sigma_0$. Thus $\varphi(\bar x)$ is equivalent, in both $\Sets$ and $\FinSets$, to the same finite Boolean combination of partition formulas and exact-cardinality sentences.

	At a finite world $W$, such partition formulas and exact-cardinality sentences have the same truth values whether $W$ is regarded as an object of $\Sets$ or of $\FinSets$. Consequently, the same is true of their finite Boolean combinations, and hence
	\begin{equation*}
		W \models_{\Sets} \varphi[\bar a]
		\qquad \text{if and only if} \qquad
		W \models_{\FinSets} \varphi[\bar a]
	\end{equation*}
	for every tuple $\bar a\in W$. The final claim about propositional modal validities follows immediately.
\end{proof}

\begin{theorem}\label{Theorem.Sets-non-empty-worlds-sentences-S5}
The propositional modal validities of a non-empty world in the category $\Sets$ of sets and functions with respect to sentential substitution instances, constitute precisely the propositional modal theory \theoryf{S5}. More precisely, for every non-empty $W$ in $\Sets$, and every first-order language $\mathscr{L}$, modal or not, such that $\Lofeq \subseteq \mathscr{L} \subseteq \LDiamondofeq$, we have
	\begin{equation*}
		\Val_{\Sets}(W, \mathscr{L}) = \theoryf{S5}.
	\end{equation*}
\end{theorem}

\begin{proof}
	To prove the upper bound, let $W$ be any world in $\Sets$. By the~\nameref{Upperboundstheorem}, it is sufficient to exhibit arbitrarily long finite dials at $W$. Fix $N\in\omega$. For $n>0$, let $d_n$ be the assertion that there are exactly $n$ elements, and let $d_0$ assert that either there are no elements or there are at least $N$ elements. Then $\<d_n\mid n<N>$ is a dial on the cone above $W$: from any world $U$ on that cone, one can reach a world satisfying $d_n$ by mapping $U$ to an $n$-element set when $n>0$, and to an infinite world when $n=0$. Hence $\Val_{\Sets}(W,\mathscr L)\subseteq\theoryf{S5}$ for every language of substitution instances $\mathscr L$.

	For the lower bound, suppose now that $W$ is non-empty. By the~\nameref{Lowerboundstheorem}, it suffices to show that the cone above $W$ is $\emptyset$-invertible. But in the category of sets and functions, every pair of non-empty worlds can access each other, so every degenerate span in the cone above $W$ can be inverted over the empty parameter set. Thus \theoryf{S5} is valid for sentential substitution instances at every non-empty world.

	Combining the lower and upper bounds yields $\Val_{\Sets}(W,\mathscr L)=\theoryf{S5}$ for every non-empty world $W$ and every language $\mathscr L$ with $\Lofeq\subseteq\mathscr L\subseteq\LDiamondofeq$.
\end{proof}

A \emph{lollipop frame} is a propositional Kripke frame composed of a single initial node, which is not accessible from any other node, accessing a cluster of mutually accessible nodes, with no additional nodes, as illustrated in the picture below. We denote the propositional modal theory of all lollipop frames, evaluated at the initial node, by $\theoryf{Lollipop}$. We shall use below that this theory is already characterized by finite lollipop frames; the elementary finite-reduction argument is included in the proof of theorem~\ref{Theorem.Sets-empty-worlds-sentences-Lollipop}.

Some authors use the term \emph{lollipop frame} for a closely related kind of frame (for example, a chain of singleton worlds capped by a top cluster). To avoid confusion, throughout this paper we reserve \emph{lollipop frame} for the one-step frame described above (an initial node accessing a nonempty cluster). The TikZ code for the picture below is adapted from original code by Joel David Hamkins and modified by the author.
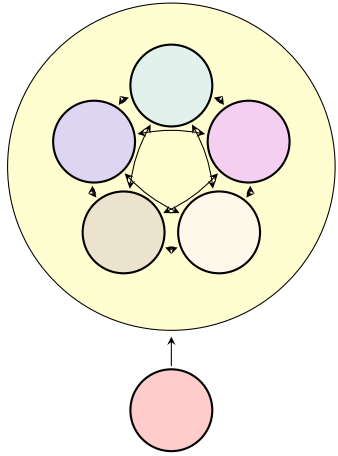
\begin{figure}[H]
	\centering
	\begin{tikzpicture}[scale=.36]
		\newcommand{\pentacluster}{
			\draw (18+72*0:1) node[circle,thick,draw,scale=3,fill=LilacMist] (0) {};
			\draw (18+72*1:1) node[circle,thick,draw,scale=3,fill=MintWhisper] (1) {};
			\draw (18+72*2:1) node[circle,thick,draw,scale=3,fill=LavenderFrost] (2) {};
			\draw (18+72*3:1) node[circle,thick,draw,scale=3,fill=VanillaHaze] (3) {};
			\draw (18+72*4:1) node[circle,thick,draw,scale=3,fill=IvoryBreeze] (4) {};
			\foreach \r/\s in {0/1,1/2,2/3,3/4,4/0} {\draw[<->,>=stealth] (\r) edge[bend right=20] (\s);}
			\foreach \r/\s in {0/2,2/4,4/1,1/3,3/0} {\draw[<->,>=stealth] (\r) edge[bend right=10] (\s);}
		}
		\begin{scope}
			\draw node[circle, thick, draw,scale=3,fill=red,fill opacity=.2] (single) {};
		\end{scope}
		\begin{scope}[shift={(0,9)}]
			\draw node[circle,draw,scale=12,fill=JDHYellow] (cluster) {};
			\begin{scope}[scale=3]
				\pentacluster
			\end{scope}
		\end{scope}
		\draw[->,>=stealth] (single) edge (cluster);
	\end{tikzpicture}
    \refstepcounter{figure}
\label{Figure.lollipop}
\end{figure}

\begin{theorem}\label{Theorem.Sets-empty-worlds-sentences-Lollipop}
    The propositional modal validities of the empty world in the category $\Sets$ of sets and functions with respect to sentential substitution instances, constitute precisely the propositional modal theory \theoryf{Lollipop}. More precisely, for every first-order language $\mathscr{L}$, modal or not, such that $\Lofeq \subseteq \mathscr{L} \subseteq \LDiamondofeq$, we have
	\begin{equation*}
		\Val_{\Sets}(\emptyset, \mathscr{L}) = \theoryf{Lollipop}.
	\end{equation*}
\end{theorem}

\begin{proof}
    We first show that the propositional modal theory \theoryf{Lollipop} is valid at the empty world $\emptyset$. By observation~\ref{Observation.Sets-and-FinSets-same}, it is sufficient to demonstrate that the propositional modal validities of the empty world in the category $\FinSets$ of finite sets and functions, with respect to sentential substitution instances, contain the propositional modal theory \theoryf{Lollipop}.

    Let $F$ be a lollipop frame with a cluster of size $\omega$. For each assertion $\sigma_m$ in the language of equality $\Lofeq$ expressing that there are precisely $m$ elements, we associate a propositional variable $p_m$. Let $w_0$ be the initial node of $F$ and the other nodes, $w_1,\ldots$, be the nodes in the cluster. Let $\bar p$ be an enumeration of the $p_m$s. Define a function $I \colon \FinSets \to F$ such that $W \mapsto w_m$ if and only if $W \models \sigma_m$. This mapping is well-defined and surjective because any world in $\FinSets$ satisfies precisely one of all $\sigma_m$s, and for each $m$, there exists a world in $\FinSets$ that satisfies $\sigma_m$. Using this, we can construct a propositional Kripke model $M$ with frame $F$ such that
	\begin{equation*}
		(M,w) \models p_m \qquad \text{if and only if} \qquad W \models \sigma_m,
	\end{equation*}
	where $w$ is the world corresponding to $W$ via the mapping $I$.

	Let $\Lambda$ be the class of all assertions obtained by closing the assertions of the form $\sigma_m$ under Boolean combinations and modal operators. By induction on the complexity of formulas, we extend the correspondence between the assertions $\sigma_m$ and the propositional variables $p_m$ to a translation between the assertions $\psi \in \Lambda$ and the propositional modal assertions $\psi^*$, as follows.
	\begin{eqnarray*}
		\sigma_m^* &=& p_{m} \\
		(\neg\psi)^*&=& \neg\psi^* \\
		(\psi_0\land\psi_1)^* &=& \psi_0^*\land\psi_1^* \\
		(\possible \psi)^* &=& \possible \psi^*
	\end{eqnarray*}

	We claim that for $\psi \in \Lambda$,
	\begin{equation*}
		(M,w) \models \psi^*(\bar p) \qquad \text{if and only if} \qquad W \models \psi,
	\end{equation*}
	where, again, $w$ is the world corresponding to $W$ via the mapping $I$. We proceed by induction on the complexity of $\psi$ simultaneously for all worlds in the cone above $W_0$---all worlds accessible from $W_0$ concurrently take the role of $W_0$. Note that although $\bar p$ is infinite, we shall use only finitely many $p_m$s for each assertion $\psi$. The case of $\sigma_m$s follows directly from the definition of the translation, and Boolean combinations go through easily.

	For the modal operator case, first suppose $W \models \possible \psi$. This means there is a world $U$ that satisfies $\psi$ and a world $w$ in $M$ corresponding to $W$ via $I$. By induction, for $u = I(U)$, we have that $(M,u) \models \psi^*(\bar p)$. By the definition of $I$, we must have that $w \leq_F u$, so $(M,w) \models \possible \psi^*(\bar p)$. Conversely, suppose $(M,w) \models \possible \psi^*(\bar p)$. Then, there is a world $W$ in $\FinSets$ such that $I(W) = w$, and there is a world $u$ in $M$ accessible from $w$ such that $(M,u) \models \psi^*(\bar p)$. But then, there is a world $U$ in $\FinSets$ such that $I(U) = u$, for which, by induction, $U \models \psi$. Finally, by the definition of $I$, the world $W$ satisfies $\possible \psi$.

	Suppose now that $W_0 \not \models \varphi(\psi_0, \ldots, \psi_n)$, where $\varphi$ is a propositional modal assertion, and $\psi_i$s are from the language of substitution instances. Due to the~\nameref{Modalityeliminationtheorem.Sets}, we can assume without loss of generality that each $\psi_i$ is in $\Lambda$. Observe from the definition of the translation that $\bigl( \varphi(\psi_0,\ldots,\psi_n) \bigr)^* = \varphi(\psi_0^*,\ldots,\psi_n^*)$. We have that $(M,w_0) \not \models \varphi(\psi_0^*,\ldots,\psi_n^*)$. However, the propositional modal theory \theoryf{Lollipop} is normal and, therefore, closed under substitutions. As a result, $\varphi$ cannot be valid at $(M,w_0)$. Since $\varphi$ was chosen arbitrarily, any statement that is not valid at $W_0$ cannot be part of the propositional modal theory \theoryf{Lollipop}. So $\theoryf{Lollipop} \of \Val_{\Sets}(\emptyset, \LDiamondofeq)$.

    \medskip

    We now show that the propositional modal validities of the empty world $\emptyset$, with respect to any assertions whatsoever, even sentential assertions in the first-order non-modal language of equality $\Lofeq$, are contained in the propositional modal theory \theoryf{Lollipop}. Let $n$ be an arbitrarily large natural number, and suppose $F$ is a lollipop frame with a cluster of size $n$. Let $w_0$ be the initial node of $F$, and let $c_1,\ldots,c_n$ be the distinct nodes in the terminal cluster. We aim to find a labeling for the frame $F$ so that we can later apply the~\nameref{LabelingLemma}. Assign the initial node $w_0$ the assertion $\Phi_0$, which says that there are no elements. For $1\le i<n$, assign the cluster node $c_i$ the assertion $\Phi_i$, which says that there are exactly $i$ distinct elements, and assign the remaining cluster node $c_n$ the assertion $\Phi_n$, which says that there are at least $n$ elements. It is not difficult to see that this assignment constitutes a labeling of $F$ for the empty world. Indeed, $\emptyset \models \Phi_0$; each world in the cone above $\emptyset$ satisfies exactly one $\Phi_i$; and for any labeled nodes $u$ and $v$ in $F$, we have $u \leq_F v$ if and only if each world in the cone above $\emptyset$ satisfying the label of $u$ accesses a world satisfying the label of $v$.

	We first record the elementary finite reduction for lollipop frames. Suppose that a propositional modal formula $\varphi$ fails at the initial node of some lollipop model. Let $\Sigma=\operatorname{Sub}(\varphi)$, and in the top cluster choose one representative world for each $\Sigma$-truth type realized there. Keeping the initial node together with these representatives yields a finite lollipop submodel. Since the accessibility pattern between the initial node and the top cluster is preserved, and since every retained top-cluster world still sees exactly the retained top cluster, an induction on formulas in $\Sigma$ shows that the truth of every formula in $\Sigma$ is preserved at all retained worlds, in particular at the initial node. Thus $\varphi$ already fails at the initial node of a finite lollipop model. Consequently, if $\theoryf{Lollipop} \not \models \varphi$, there exists a propositional Kripke model $M$ with a finite lollipop frame where $\varphi$ fails at an initial world $w_0$. By the~\nameref{LabelingLemma}, there is an assignment of the propositional variables $p_i$ to assertions $\psi_{p_i}$ in the language of substitutions such that
\begin{equation*}
	(M,w_0) \models \varphi(p_0,\ldots,p_k)
	\qquad\text{if and only if}\qquad
	\emptyset \models \varphi(\psi_{p_0},\ldots,\psi_{p_k}).
\end{equation*}
Since $\varphi$ is false at $w_0$, it follows that $\emptyset \models \neg \varphi(\psi_{p_0},\ldots,\psi_{p_k})$, and therefore $\varphi$ is not valid at the empty world. As $\varphi$ was arbitrary, this shows
\begin{equation*}
	\Val_{\Sets}(\emptyset, \Lofeq) \subseteq \theoryf{Lollipop}.
\end{equation*}
Combining this with the lower bound already proved at the largest sentential language, we obtain for every first-order language $\mathscr L$ with $\Lofeq\subseteq \mathscr L\subseteq \LDiamondofeq$ the chain
\begin{equation*}
	\theoryf{Lollipop}
	\subseteq
	\Val_{\Sets}(\emptyset,\LDiamondofeq)
	\subseteq
	\Val_{\Sets}(\emptyset,\mathscr L)
	\subseteq
	\Val_{\Sets}(\emptyset,\Lofeq)
	\subseteq
	\theoryf{Lollipop}.
\end{equation*}
Hence $\Val_{\Sets}(\emptyset,\mathscr L)=\theoryf{Lollipop}$, as claimed.
\end{proof}

\begin{theorem}\label{Theorem.Sets-infinite-worlds-formulas-S4.2}
    The propositional modal validities of an infinite world in the category $\Sets$ of sets and functions with respect to formulaic substitution instances, with parameters, constitute precisely the propositional modal theory \theoryf{S4.2}. More precisely, for every infinite $W$ in $\Sets$, and for every first-order language $\mathscr{L}$, modal or not, such that $\Lofeq \subseteq \mathscr{L} \subseteq \LDiamondofeq$, we have
	\begin{equation*}
		\Val_{\Sets}(W, \mathscr{L}_{W}) = \theoryf{S4.2}.
	\end{equation*}
\end{theorem}

\begin{proof}
Let $W$ be an infinite world in $\Sets$.

For the lower bound, by the~\nameref{Lowerboundstheorem}, it suffices to show that the cone above $W$ is $W$-amalgamable. So let $f\colon W\to U$ witness that $U$ lies on the cone above $W$, and consider any span
\begin{equation*}
	U_0 \xleftarrow{f_0} U \xrightarrow{f_1} U_1
\end{equation*}
on that cone. Choose a singleton world $\singleton{s}$. The constant maps $g_0\colon U_0\to \singleton{s}$ and $g_1\colon U_1\to \singleton{s}$ complete the span to a square, and since both composites $g_0\compose f_0\compose f$ and $g_1\compose f_1\compose f$ land in a singleton, they agree on every element of $W$. Thus the square commutes over the parameter set $W$, and the cone above $W$ is $W$-amalgamable. Hence $\theoryf{S4.2}\subseteq \Val_{\Sets}(W,\mathscr L_W)$ for every language $\mathscr L$ with $\Lofeq\subseteq\mathscr L\subseteq\LDiamondofeq$.

For the upper bound, by statement~(4) of the~\nameref{Upperboundstheorem}, it suffices to show that $W$ admits arbitrarily large finite families of independent unpushed buttons together with arbitrarily long dials, in a way that is independent of those buttons. Since all the control statements we use are $\Lofeq$-assertions with parameters from $W$, this will yield $\Val_{\Sets}(W,\mathscr L_{W})\subseteq \theoryf{S4.2}$.

\medskip

Fix $N\in\omega$, and choose pairwise distinct parameters
$u_1,v_1,\ldots,u_N,v_N\in W$.
Let $\rho(\bar u,\bar v)$ be the assertion saying that some cross-wiring occurs among these
parameters, namely:
\begin{equation*}
\bigvee_{j\neq k}\Bigl(u_j=u_k \lor u_j=v_k \lor v_j=v_k\Bigr).
\end{equation*}
For each $1\le i\le N$, define $b_i$ to say $(u_i=v_i) \lor \rho$.
At $W$ we have $u_i\neq v_i$ and $\rho$ is false (all parameters are distinct), hence each $b_i$
is unpushed at $W$.

Each $b_i$ is a button. Indeed, from any world above $W$, we can further collapse the images of
$u_i$ and $v_i$ to make $u_i=v_i$ true. And once $u_i=v_i$ or $\rho$ holds at some world, it
remains true at all further accessible worlds because equality is preserved by functions. Thus
$b_i$ becomes necessary once it becomes true.

To see independence, fix any $S\subseteq\{1,\ldots,N\}$. Define a function
$f_S$ that identifies $u_i$ with $v_i$ exactly when $i\in S$, and is the identity
on all other elements (in particular, it does not identify any $u_j$ with any $v_k$ for $j\neq k$
and does not identify distinct $v$'s with each other). Then at the target world, $\rho$ is still
false and $u_i=v_i$ holds if and only if $i\in S$. Consequently, exactly the buttons $b_i$ with
$i\in S$ are pushed there, and the others remain unpushed. Since $N$ was arbitrary, this gives
arbitrarily large finite independent families of unpushed buttons at $W$.

\medskip

Fix $M\in\omega$ and let $\bar a$ be the tuple listing all the parameters
$\bar a=(u_1,v_1,\ldots,u_N,v_N)$.
For each $n<M$, let $d_n(\bar x)$ assert that there are exactly $n$ elements distinct from all
entries of $\bar x$, and let $d_M(\bar x)$ assert that there are at least $M$ such elements. Then, at every world in the cone above $W$, exactly one of the assertions $d_0(\bar a),\ldots,
d_M(\bar a)$ holds (depending on the number of elements outside the parameter tuple), and each
dial value is reachable from any world above $W$.

Moreover, we can change the dial value without affecting the buttons $b_i$. Indeed, suppose
$h\colon W\to V$ is any function and we want to realize a specific dial value
$n\le M$ above $V$ while keeping the pattern of equalities among the parameters fixed. Choose a
target world $U=h(\bar a) \cup C$,
where $C$ is disjoint from $h(\bar a)$ and has size $n$ if $n<M$, and has size $M$ if $n=M$
(so that $U$ has at least $M$ elements outside $h(\bar a)$). Define $g\colon V\to U$ to fix every
element of $h(\bar a)$ pointwise and to send $V\setminus h(\bar a)$ to a single chosen element of
$C$ (or, if $n=0$, to a single chosen element of $h(\bar a)$). Then $U$ satisfies the desired dial
value, and because $g$ fixes $h(\bar a)$ pointwise it does not introduce any new equalities among
the parameters; in particular, it does not change the truth values of the buttons $b_i$. This is
exactly the required independence.

\medskip

Since $N$ and $M$ were arbitrary, statement~(4) of the~\nameref{Upperboundstheorem} applies and we
conclude $\Val_{\Sets}(W,\mathscr L_{W})\ \of\ \theoryf{S4.2}$.
Together with $\theoryf{S4.2}\of\Val_{\Sets}(W,\mathscr L_{W})$, we obtain $\Val_{\Sets}(W,\mathscr L_{W})=\theoryf{S4.2}$,
as desired.
\end{proof}

Restricting to the full subcategory $\InfSets$ changes the picture. The upper-bound argument above for $\Sets$ uses a dial obtained from finite quotients; that dial disappears in $\InfSets$, and only the finite partition data carried by the parameters remain. By corollary~\ref{Corollary.InfSets-sentences-Triv}, the sentential validities in $\InfSets$ are trivial, so the interesting statement here is formulaic.

\begin{theorem}\label{Theorem.InfSets-formulas-Grz.2}
    The propositional modal validities of a world in the category $\InfSets$ of infinite sets and functions with respect to formulaic substitution instances, with parameters, constitute precisely the propositional modal theory \theoryf{Grz.2}. More precisely, for every world $W$ in $\InfSets$, and for every first-order language $\mathscr{L}$, modal or not, such that $\Lofeq \subseteq \mathscr{L} \subseteq \LDiamondofeq$, we have
    \begin{equation*}
        \Val_{\InfSets}(W,\mathscr{L}_{W})=\theoryf{Grz.2}.
    \end{equation*}
\end{theorem}

\begin{proof}
    Let $W$ be a world in $\InfSets$.

    For the upper bound, by statement~(2) of the~\nameref{Upperboundstheorem2}, it suffices to show that $W$ has arbitrarily large finite independent families of unpushed buttons. Use the same construction as in the proof of theorem~\ref{Theorem.Sets-infinite-worlds-formulas-S4.2}: fix $N\in\omega$, choose pairwise distinct parameters
    \begin{equation*}
        u_1,v_1,\dots,u_N,v_N\in W,
    \end{equation*}
    let $\rho(\bar u,\bar v)$ assert that some unwanted cross-identification occurs among these parameters, and define
    \begin{equation*}
        b_i := (u_i=v_i)\lor \rho(\bar u,\bar v)
        \qquad (1\le i\le N).
    \end{equation*}
    Exactly as before, the assertions $b_1,\dots,b_N$ form an independent family of buttons at $W$. The only extra point is that the quotient worlds used there remain objects of $\InfSets$; but this is immediate, since identifying only finitely many points of an infinite set still yields an infinite quotient. Therefore $W$ has arbitrarily large finite independent families of unpushed buttons, and so
    \begin{equation*}
        \Val_{\InfSets}(W,\mathscr{L}_{W})\subseteq \theoryf{Grz.2}.
    \end{equation*}

    For the lower bound, we first obtain $\theoryf{S4.2}$. By the~\nameref{Lowerboundstheorem}, it suffices to show that the cone above $W$ is $W$-amalgamable. So let
    \begin{equation*}
        U_0 \xleftarrow{f_0} U \xrightarrow{f_1} U_1
    \end{equation*}
    be any span in the cone above $W$, and let $f\colon W\to U$ witness that $U$ lies on that cone. Choose any infinite set $I$ and any element $c\in I$. The constant maps
    \begin{equation*}
        g_0\colon U_0\to I,
        \qquad
        g_1\colon U_1\to I
    \end{equation*}
    with value $c$ complete the span to a square, and since both composites
    \begin{equation*}
        g_0\compose f_0\compose f
        \qquad\text{and}\qquad
        g_1\compose f_1\compose f
    \end{equation*}
    are constant, the square commutes over the parameter set $W$. Hence the cone above $W$ is $W$-amalgamable, and so
    \begin{equation*}
        \theoryf{S4.2}\subseteq \Val_{\InfSets}(W,\mathscr{L}_{W}).
    \end{equation*}

    To obtain \theoryf{Grz.2}, we verify the hypothesis of the~\nameref{Lowerboundstheorem2}. Let $U$ be any world in the cone above $W$, witnessed by some accessibility mapping $f\colon W\to U$. Let $\bar w=(w_0,\dots,w_{n-1})$ be a finite tuple from $W$, and suppose that the assertion $\varphi[f(\bar w)]$, with $\varphi(\bar x)$ in $\mathscr{L}$, is contingent at $U$.

    By corollary~\ref{Corollary.InfSets-formulas-partitions}, specialized to $\InfSets$, there is a set $X$ of partitions of $\{0,\dots,n-1\}$ such that, at every world accessible from $U$, the assertion $\varphi[f(\bar w)]$ holds exactly when the tuple induced there realizes a partition belonging to $X$. In particular, the truth of $\varphi[f(\bar w)]$ depends only on the partition of $\{0,\dots,n-1\}$ induced by equality among the images of the tuple.

    Let $p$ be the partition currently induced by $f(\bar w)$ at $U$, and let $\Pi_p$ be the finite lattice of all coarsenings of $p$. For each $q\in\Pi_p$, choose a quotient map
    \begin{equation*}
        h_q\colon U\twoheadrightarrow U_q
    \end{equation*}
    that identifies exactly the finitely many points required to realize $q$ on the tuple $f(\bar w)$ and makes no other identifications. Since only finitely many points are merged, each $U_q$ is still infinite. By the previous paragraph, the truth value of $\varphi[h_q\compose f(\bar w)]$ depends only on $q$, so the set
    \begin{equation*}
        T:=\{q\in \Pi_p \mid U_q\models \varphi[h_q\compose f(\bar w)]\}
    \end{equation*}
    is well defined. Because $p$ is realized already at $U$ and $\varphi[f(\bar w)]$ is contingent there, some coarsening of $p$ has the opposite truth value. Hence $T$ is neither empty nor all of $\Pi_p$. Let $\top$ be the top element of $\Pi_p$, let $t\in\{0,1\}$ be the truth-value of $\varphi$ at $U_{\top}$, and choose a partition $q\in\Pi_p$ that is maximal among those with truth-value $1-t$.

    Suppose first that $1-t=1$. Then $U_q\models\varphi[h_q\compose f(\bar w)]$, while $U_{\top}\models\neg\varphi[h_{\top}\compose f(\bar w)]$. Since $\top$ is a strict coarsening of $q$, the world $U_q$ can access a world where $\varphi$ is false, so $\varphi$ is possibly false at $U_q$. Now let $V$ be any world accessible from $U_q$ where $\varphi$ is false, and let $q'$ be the partition realized there by the image of $\bar w$. Then $q'\neq q$, since the truth of $\varphi$ depends only on the realized partition and $U_q\models\varphi[h_q\compose f(\bar w)]$ while $V\models\neg\varphi$. Thus $q'$ is a strict coarsening of $q$, and by maximality of $q$ it has truth-value $t=0$. Any further world accessible from $V$ realizes a coarsening of $q'$, and therefore again a strict coarsening of $q$; by maximality, $\varphi$ is false there as well. Thus once $\varphi$ becomes false above $U_q$, it remains false forever, so $\varphi$ is penultimate at $U_q$.

    Suppose instead that $1-t=0$. Then $U_q\models\neg\varphi[h_q\compose f(\bar w)]$, while $U_{\top}\models\varphi[h_{\top}\compose f(\bar w)]$. Again $\top$ is a strict coarsening of $q$, so $U_q$ can access a world where $\neg\varphi$ is false, and hence $\neg\varphi$ is possibly false at $U_q$. If $V$ is any world accessible from $U_q$ where $\neg\varphi$ is false, let $q'$ be the partition realized there. Then $q'\neq q$, since the truth of $\varphi$ depends only on the realized partition and $U_q\models\neg\varphi[h_q\compose f(\bar w)]$ while $V\models\varphi$. Thus $q'$ is a strict coarsening of $q$, so by maximality of $q$ it has truth-value $t=1$. Every further world accessible from $V$ again realizes a strict coarsening of $q$, and therefore $\neg\varphi$ remains false there. Thus $\neg\varphi$ is penultimate at $U_q$.

    In either case, the contingency $\varphi[f(\bar w)]$ possibly attains a penultimate truth-value at some further world above $U$. Since $U$ and $\varphi[f(\bar w)]$ were arbitrary, the hypothesis of the~\nameref{Lowerboundstheorem2} is satisfied. Therefore
    \begin{equation*}
        \theoryf{Grz.2}\subseteq \Val_{\InfSets}(W,\mathscr{L}_{W}).
    \end{equation*}

    Combining the upper and lower bounds yields
    \begin{equation*}
        \Val_{\InfSets}(W,\mathscr{L}_{W})=\theoryf{Grz.2},
    \end{equation*}
    as desired.
\end{proof}

\begin{theorem}\label{Theorem.Sets-singleton-world-formulas-S5}
    The propositional modal validities of a singleton world in the category $\Sets$ of sets and functions with respect to formulaic substitution instances, with parameters, constitute precisely the propositional modal theory \theoryf{S5}. More precisely, for every $\singleton{s}$ in $\Sets$, and for every first-order language $\mathscr{L}$, modal or not, such that $\Lofeq \subseteq \mathscr{L} \subseteq \LDiamondofeq$, we have
	\begin{equation*}
		\Val_{\Sets}(\singleton{s}, \mathscr{L}_{\singleton{s}}) = \theoryf{S5}.
	\end{equation*}
\end{theorem}

\begin{proof}
    Since $\singleton{s}$ is non-empty, theorem~\ref{Theorem.Sets-non-empty-worlds-sentences-S5} gives $\Val_{\Sets}(\singleton{s},\mathscr L)=\theoryf{S5}$ for every language $\mathscr L$ with $\Lofeq\subseteq \mathscr L\subseteq \LDiamondofeq$. Because $\mathscr L\subseteq \mathscr L_{\singleton{s}}$, monotonicity of validities yields the upper bound
\begin{equation*}
	\Val_{\Sets}(\singleton{s}, \mathscr{L}_{\singleton{s}}) \subseteq \Val_{\Sets}(\singleton{s}, \mathscr{L}) = \theoryf{S5}.
\end{equation*}

    For the reverse inclusion, by the~\nameref{Lowerboundstheorem} it suffices to show that the cone above
    $\singleton{s}$ is invertible (or equivalently, $\{s\}$-invertible). That is, given morphisms $f\colon \singleton{s}\to W$ and $g\colon W\to U$ in the cone above $\singleton{s}$, we must find a morphism $h\colon U\to W$ (also in that cone) such that $(h\compose g)(f(s))=f(s)$.
    Since $W$ is nonempty (it contains $f(s)$), we can define a function $h\colon U\to W$ by stipulating $h\bigl(g(f(s))\bigr)=f(s)$,
    and defining $h$ arbitrarily on the remaining elements of $U$. Then $(h\compose g)(f(s))=h\bigl(g(f(s))\bigr)=f(s)$,
    as required. Hence the cone above $\singleton{s}$ is invertible, and therefore $\theoryf{S5}\subseteq\Val_{\Sets}(\singleton{s}, \mathscr{L}_{\singleton{s}})$.
    Together with the upper bound, we obtain $\Val_{\Sets}(\singleton{s}, \mathscr{L}_{\singleton{s}})=\theoryf{S5}$.
\end{proof}

A \emph{partition prelattice} of a set $P$ is obtained from the partition lattice of $P$ by replacing each element with a cluster of one or more equivalent elements, all related by the partition lattice relation. Equivalently, it is a partial preorder ${\leq}$ (a reflexive and transitive relation) on a set $Q$, such that the quotient of $Q$ by the equivalence relation $p \equiv q \iff p \leq q \leq p$ is the partition lattice of $P$ under the induced quotient relation ${\leq}$.

The propositional modal theory $\theoryf{Prepartition}_n$ is the smallest normal propositional modal theory valid in frames that are partition prelattices over an $n$-element set. It is characterized already by the finite such frames. Indeed, if a formula fails at a node of some partition prelattice over an $n$-element set, then in each cluster one may keep one representative for each truth type over the finitely many subformulas of the formula. Since the quotient partition lattice is finite for fixed $n$, the retained frame is a finite partition prelattice over the same $n$-element set, and the usual induction on subformulas preserves the truth of the given formula at the retained node.

\begin{theorem}\label{Theorem.Sets-world-size-n-formulas-Prepartition}
    The propositional modal validities of a world of finite size $n>0$ in the category $\Sets$ of sets and functions with respect to formulaic substitution instances, with parameters, constitute precisely the propositional modal theory $\theoryf{Prepartition}_n$. More precisely, for every $n>0$ and every $W_0$ in $\Sets$ such that $|W_0| = n$, and for every first-order language $\mathscr{L}$, modal or not, such that $\Lofeq \subseteq \mathscr{L} \subseteq \LDiamondofeq$, we have
	\begin{equation*}
		\Val_{\Sets}(W_0, {\mathscr{L}}_{W_0}) = \theoryf{Prepartition}_n.
	\end{equation*}
\end{theorem}

\begin{proof}
    We first show that $\theoryf{Prepartition}_n$ is valid at any $n$-element world of
	$\Sets$ with $n>0$ for formulaic assertions with parameters, already in the modal language
	$\LDiamondofeq$. By observation~\ref{Observation.Sets-and-FinSets-same}, it is enough to
	work in $\FinSets$.

	Let $W_0$ be a world of size $n>0$ in the category $\FinSets$, and let $F$ be a propositional Kripke frame that is the prepartition prelattice of $n$ with each cluster having a size of $\omega$. Each world is determined by a partition and its cluster index $m < \omega$. For every atomic assertion $a_i = a_j$ in the language of equality $\Lofeq(a_i \mid i<n)$, where $a_i$s are added as constants, we associate a propositional variable $p_{ij}$. Additionally, for every assertion $\sigma_m(\bar a)$ expressing that there are $m$ elements different from the ones in $\bar a$, we associate a propositional variable $q_m$. Let $\bar p$ be an enumeration of the $p_{ij}$s and $\bar q$ be an enumeration of the $q_m$s. We now construct a propositional Kripke model $M$ with a frame of $F$, so that for any world $w \in M$, the following conditions hold.
	\begin{enumerate}
		\item $(M, w) \models p_{ij}$ just in case $i$ and $j$ belong to the same block in the partition represented by $w$;
		\item $(M,w) \models q_m$ just in case the cluster index of the world $w$ is $m$.
	\end{enumerate}
	In this manner, the truth value pattern of $\bar p$ and $\bar q$ uniquely encodes the partition $w$ as well as its cluster index. Let $w_0$ be the initial node of $F$, representing a trivial partition with all blocks of size one and a cluster index of zero. We have that $(M,w_0) \models p_{ij}$ if and only if $i = j$, and $(M,w_0) \models q_{m}$ if and only if $m = 0$.

	World $W_0$ has a model expansion $W_0^{+}$ to the language $\Lofeq(a_i \mid i<n)$, where each constant $a_i$ is interpreted as a distinct individual of $W_0$. Any function $f \colon W_0 \to W$ results in a model expansion $W^f$ of $W$ to the language $\Lofeq(a_i \mid i<n)$, where each constant $a_i$ is interpreted as the image under $f$ of the individual of $W_0$ that would be interpreted as $a_i$ in the expansion $W_0^{+}$. Any such $f$ is a map that identifies some elements in $W_0$ in its image while maintaining distinctness among other elements. In fact, $f$ determines a partition of $n$, and thus a cluster of $F$. More specifically, two elements belong to the same block of the partition if and only if they are identified by $f$. Furthermore, $f$ determines a node in that partition. This node is determined by the number of elements of $W$ other than those in the set of values of the $a_i$s, that is outside $f[W_0]$. As a result, there is a natural surjective correspondence, $W^f \mapsto w_f$, between model expansions of the form $W^f$ for worlds $W$ in the cone above $W_0$ and worlds $w_f$ in $M$, such that:
	\begin{enumerate}
		\item $(M,w_f) \models p_{ij}$ just in case $W^f \models a_i = a_j$, and
		\item $(M,w_f) \models q_{m}$ just in case $W^f \models \sigma_m(\bar a)$.
	\end{enumerate}
	This implies that given $W^f \mapsto w_f$ and $U^g \mapsto u_g$, the partition represented by $u_g$ is a coarsening of the partition represented by $w_f$, and hence $u_g$ is accessible from $w_f$, if and only if $g$ determines a partition of $n$ that is coarser than the one determined by $f$. Consequently, $g$ factors through $f \colon W_0 \to W^f$, that is, $g = h \compose f$ for some $h \colon W^f \to U^g$, and $W^f$ accesses $U^g$ if and only if $w_f$ accesses $u_g$.

	Let $\Lambda$ be the class of atomic assertions $\psi(\bar a)$ and assertions $\sigma_n(\bar a)$ in $\LDiamondofeq(a_i \mid i<n)$, closed under Boolean connectives and modal operators. We proceed by induction on the complexity of formulas to extend the correspondences between the atomic assertions $a_i = a_j$ and the propositional variables $p_{ij}$, as well as the existential assertions $\sigma_m(\bar a)$ and the propositional variables $q_m$, to a translation between the assertions $\psi \in \Lambda$ and the propositional modal assertions $\psi^*$, so that
	\begin{eqnarray*}
		(a_i = a_j)^* &=& p_{ij} \\
		(\neg\psi)^*&=& \neg\psi^* \\
		(\psi_0\land\psi_1)^* &=& \psi_0^*\land\psi_1^* \\
		(\possible \psi)^* &=& \possible \psi^* \\
		(\sigma_m(\bar a))^* &=& q_m
	\end{eqnarray*}
	We claim that, for $\psi \in \Lambda$,
	\begin{equation*}
		(M,w_f) \models \psi^*(\bar p, \bar q) \text{ if and only if } W^f \models\psi(\bar a).
	\end{equation*}

	We proceed by induction on the complexity of $\psi$ simultaneously for all worlds $W$ in the cone above $W_0$ and all functions $g \colon W \to U$---all $W$s concurrently take the role of $W_0$, and all $g$s concurrently take the role of $f$. Note that although $\bar q$ is infinite, we will use only finitely many $q_m$s for each assertion $\psi$. The atomic case and the case of $\sigma_m(\bar a)$s follow directly from the definition of the translation. Boolean combinations also go through easily.

	Consider the modal operator case. Suppose $W^f \models \possible \psi(\bar a)$. Then there exists a world $w_f$ in $M$ that naturally corresponds to $W^f$, and a function $g \colon W^f \to U^{g \compose f}$ such that $U^{g \compose f} \models \psi(\bar a)$. Here, $W^f$ and $U^{g \compose f}$ are the respective model expansions of worlds $W$ and $U$ to the language $\LDiamondofeq(a_i \mid i<n)$. Now, there exists a world $u_{g \compose f}$ in $M$ that naturally corresponds to $U^{g \compose f}$. By induction, we have $(M,u_{g \compose f}) \models \psi^*(\bar p, \bar q)$. From our earlier discussion, since $W^f$ accesses $U^{g \compose f}$, it follows that $w_f$ accesses $u_{g \compose f}$. Thus, $(M,w_f) \models \possible \psi^*(\bar p, \bar q)$ as desired. Conversely, due to the surjectivity of the natural correspondence between model expansions in $\mathrm{Cone}(W_0)$ and worlds in $M$, for any world $w$ in $M$, there exists a world $W$ in the cone above $W_0$ and a function $f \colon W_0 \to W$ such that the model expansion $W^f$ is associated with $w$; we can now call it $w_f$. Suppose $(M,w_f) \models \possible \psi^*(\bar p, \bar q)$. There exists a world $u$ in $M$ such that $u$ is accessible from $w_f$ and $(M,u) \models \psi^*(\bar p, \bar q)$. Once again, due to the surjectivity of the natural correspondence, there exists a corresponding model expansion $U^g$ associated with $u$, and we can refer to $u$ as $u_g$. Since $w_f$ accesses $u_g$, as mentioned earlier, we can deduce that $W^f$ accesses $U^g$. This ultimately leads to the conclusion that $W^f \models \possible \psi(\bar a)$.

	Finally, suppose that $W_0 \not \models \varphi(\psi_0, \ldots, \psi_n)$, where $\varphi$ is a propositional modal assertion, and $\psi_i$s are from the language of substitution instances. Due to the~\nameref{Modalityeliminationtheorem.Sets}, we can assume, without loss of generality, that each $\psi_i$ belongs to $\Lambda$. Observe from the definition of the translation that $\bigl( \varphi(\psi_0,\ldots,\psi_n) \bigr)^* = \varphi(\psi_0^*,\ldots,\psi_n^*)$. It follows that $(M,w_0) \not \models \varphi(\psi_0^*,\ldots,\psi_n^*)$. However, the propositional modal theory $\theoryf{Prepartition}_n$ is normal and, therefore, closed under substitutions. Consequently, $\varphi$ cannot be valid at $(M,w_0)$. Since $\varphi$ was chosen arbitrarily, any statement that is not valid at $W_0$ cannot be in the propositional modal theory $\theoryf{Prepartition}_n$. And so, for any non-empty world $W_0$ in $\Sets$ such that $|W_0|=n$, we have $\theoryf{Prepartition}_n \of \Val_{\Sets}(W_0, {\LDiamondofeq}_{W_0})$.

    \medskip

    We now show that the propositional modal validities of any world of size $n > 0$ are contained in the propositional modal theory $\theoryf{Prepartition}_n$, with respect to formulaic substitutions. By the~\nameref{LabelingLemma}, it suffices to label an arbitrary finite prepartition prelattice $F$ of $n$. For each partition $P$ of $\{0,\ldots,n-1\}$, write the finite cluster over $P$ as
		\begin{equation*}
			C_P=\{(P,0),\ldots,(P,r_P)\},
		\end{equation*}
		where $r_P$ may depend on $P$. For $k<r_P$, let $\Phi_{(P,k)}$ say that $P[\bar a]$ holds and that there are exactly $k$ elements other than the values of $\bar a$; let $\Phi_{(P,r_P)}$ say that $P[\bar a]$ holds and that there are at least $r_P$ such elements. Thus the final node of each cluster is its own local ``at least'' bucket.

		These labels partition the cone above $W_0$: the image of $\bar a$ determines a unique partition $P$, and the number of elements outside that image is either exactly $k<r_P$ or at least $r_P$. Also $W_0$ satisfies the label of the initial node.
		Now fix $w=(P,k)$ and $v=(Q,l)$ in $F$. If $w\le_F v$, then $Q$ is a coarsening of $P$. From any world satisfying $\Phi_w$, choose a target set $U=B\cup C$, where $B$ has one point for each block of $Q$ and $C$ is disjoint from $B$, of size $l$ if $l<r_Q$ and of size $r_Q$ if $l=r_Q$. Map each distinguished element named by $a_i$ to the point of $B$ corresponding to the $Q$-block of $i$, and send every other element arbitrarily into $B$. The target then satisfies $\Phi_v$, and hence $\Phi_w\vdash\possible\Phi_v$ in the cone above $W_0$. Conversely, if a world satisfying $\Phi_w$ can access a world satisfying $\Phi_v$, then the equality pattern of the parameter tuple can only coarsen, so $Q$ is a coarsening of $P$, equivalently $w\le_F v$.
		Therefore $w\mapsto\Phi_w$ is a labeling, and by the Labeling Lemma we conclude that
	\begin{equation*}
		\Val_{\Sets}(W_0,{\Lofeq}_{W_0})\subseteq \theoryf{Prepartition}_n.
	\end{equation*}
	Combining this with the lower bound already established at the largest language of formulaic substitution instances, we obtain for every first-order language $\mathscr L$ with $\Lofeq\subseteq \mathscr L\subseteq \LDiamondofeq$ the chain
	\begin{align*}
		\theoryf{Prepartition}_n
		&\subseteq \Val_{\Sets}(W_0,{\LDiamondofeq}_{W_0})
		\subseteq \Val_{\Sets}(W_0,{\mathscr L}_{W_0})\\
		&\subseteq \Val_{\Sets}(W_0,{\Lofeq}_{W_0})
		\subseteq \theoryf{Prepartition}_n.
	\end{align*}
	Hence $\Val_{\Sets}(W_0,{\mathscr L}_{W_0})=\theoryf{Prepartition}_n$, as claimed.
\end{proof}

\par\smallskip
Although no individual world in $\FinSets$ has formulaic validities exactly $\theoryf{S4.2}$, the category as a whole does: one has $\bigcap_{n<\omega}\theoryf{Prepartition}_n=\theoryf{S4.2}$.

\section{The modal theory of the category of sets and surjective functions}\label{Section.category-of-sets-with-surjections}

We now consider surjective functions on $\Sets[$\onto$]$, together with the full subcategories $\FinSets[$\onto$]$ and $\InfSets[$\onto$]$. Since surjective functions from finite sets cannot map onto infinite sets, any surjective function originating from a finite world must target another finite world. Therefore, the worlds accessible from a finite world via surjections in $\Sets[$\onto$]$ are all finite, making the cones above finite worlds in $\Sets[$\onto$]$ and $\FinSets[$\onto$]$ identical. This leads to the observation below.

\begin{observation}\label{Observation.Surj-and-FinSurj-same}
	A finite world in the category of sets and surjective functions has precisely the same modal truths as that same world considered in the category of finite sets and surjective functions. More precisely, if $W$ is a finite world, then
	\begin{equation*}
		W \models_{\Sets[$\onto$]} \varphi[\bar x] \qquad \text{if and only if} \qquad W \models_{\FinSets[$\onto$]} \varphi[\bar x],
	\end{equation*}
	for any assertion $\varphi$ in the modal language of equality $\LDiamondofeq$, with parameters or not. In particular, any world in the category $\FinSets[$\onto$]$ has the same propositional modal validities as it would have in the category $\Sets[$\onto$]$.
\end{observation}

\begin{theorem}\label{Theorem.Surj-singleton-or-empty-world-Triv}
    The propositional modal validities of a world of size at most one in the category $\Sets[$\onto$]$ of sets and surjective functions constitute precisely the trivial propositional modal theory \theoryf{Triv}. More precisely, for every world $W$ in $\Sets[$\onto$]$ that is either a singleton world $\singleton{s}$ or the empty world $\emptyset$, and every first-order language $\mathscr{L}$, modal or not, such that $\Lofeq \of \mathscr{L} \of {\LDiamondofeq}_{W}$, we have that
    \begin{equation*}
		\Val_{\Sets[$\onto$]}(W, \mathscr{L}) = \theoryf{Triv}.
	\end{equation*}
\end{theorem}

\begin{proof}
	The propositional modal validities of a world in a Kripke category are always contained in the trivial propositional modal theory \theoryf{Triv}. From theorem~\ref{Theorem.world-complete-iff-modalities-trivialize}, it is sufficient to show that the cones above the singleton worlds $\singleton{s}$ and the empty world $\emptyset$ are model complete. In other words, we aim to demonstrate that the surjection $f \colon W \to U$ in the cones are $\Lofeq$-truth preserving.

	The cone above the empty world satisfies the requirements, because the only surjection in it is the empty function $\emptyset \colon \emptyset \to \emptyset$, which is of course $\Lofeq$-truth preserving. Similarly, the cone above any singleton world $\singleton{s}$ meets the desired conditions, as the only surjections in it are bijections between the singletons $\pi \colon \singleton{u} \to \singleton{v}$, which are $\Lofeq$-truth preserving by the \nameref{Renaming-lemma}.
\end{proof}

\begin{lemma}\label{Lemma.Surj-worlds-contain-Grz.3}
	The propositional modal theory \theoryf{Grz.3} is valid at every world in the category
	$\Sets[$\onto$]$ of sets and surjective functions, for all sentential assertions, even in the
	first-order modal language of equality $\LDiamondofeq$. In other words, for any world $W$ in
	$\Sets[$\onto$]$, we have
	\begin{equation*}
		\theoryf{Grz.3} \of \Val_{\Sets[$\onto$]}(W, \LDiamondofeq).
	\end{equation*}
\end{lemma}

\begin{proof}
	By theorem~\ref{Theorem.Surj-singleton-or-empty-world-Triv}, the empty world as well as every
	singleton world validate the trivial propositional modal theory \theoryf{Triv}. Consequently,
	they also validate \theoryf{Grz.3}.

    \medskip

	Let $W$ be an arbitrary non-empty world in the Kripke category $\Sets[$\onto$]$ of sets with
	surjective functions.

	We first verify the validity of $\theoryf{S4.3}$ on the cone above $W$.
	Fix an arbitrary world $W'$ in that cone. By the \nameref{Modalityeliminationtheorem.Sets}, every
	$\LDiamondofeq$-sentence is equivalent in $\Sets[$\onto$]$ to a non-modal $\Lofeq$-sentence. In the
	pure language of equality, such a sentence is determined solely by whether the underlying set has
	exactly $n$ elements for some finite $n$, or is infinite. In particular, all infinite worlds agree
	on the truth of every sentential assertion.

	Suppose towards a contradiction that there are independent weak buttons $b_0$ and $b_1$ at $W'$.
	Then there exist worlds $U_0,U_1$ accessible from $W'$ such that
	\begin{equation*}
		U_0\models \necessary b_0\land \neg b_1
		\qquad\text{and}\qquad
		U_1\models \necessary b_1\land \neg b_0.
	\end{equation*}
	If both $U_0$ and $U_1$ are finite, say of sizes $m_0$ and $m_1$, then either $m_0\ge m_1$ or
	$m_1\ge m_0$ as natural numbers. Without loss of generality, assume $m_0\ge m_1$. Then there is a
	surjection from $U_0$ onto a world $V$ of size $m_1$. Since sentential truth depends only on
	cardinality, $V$ satisfies exactly the same sentential assertions as $U_1$, and in particular
	$V\models \neg b_0$. But $U_0\models \necessary b_0$, contradiction.

	If one of $U_0,U_1$ is infinite and the other finite, then the infinite one surjects onto a world
	of the finite cardinality of the other, and the same contradiction arises because sentential truth
	depends only on whether the world is infinite or on its exact finite size. Finally, if both are
	infinite, then they agree on all sentential assertions, contradicting the fact that one satisfies
	$\neg b_0$ while the other satisfies $\necessary b_0$ (and hence, by \axiomf{T}, satisfies $b_0$).

	Thus there are no independent weak buttons at $W'$. Since $W'$ was arbitrary in the cone above $W$,
	lemma~\ref{Lemma.S4.3-valid-iff-no-independent-buttons} yields that $\theoryf{S4.3}$ is valid at
	every world in the cone above $W$.

	\smallskip

	We now verify the validity of \theoryf{Grz} on the cone above $W$.
	Fix again an arbitrary world $W'$ in the cone above $W$, and let $\varphi$ be any
	$\LDiamondofeq$-sentence that is contingent at $W'$. As above, the truth of $\varphi$ depends only
	on the exact finite size of the world, or on infinitude. Let $t$ be the truth value of $\varphi$
	in a singleton world $1$.

	Since $\varphi$ is contingent at $W'$, there is some accessible world from $W'$ where $\varphi$
	has truth value $\neg t$. Because every non-modal $\Lofeq$-sentence is equivalent to a finite Boolean
	combination of the sentences $\sigma_k$ asserting that there are exactly $k$ elements, there is some
	threshold $N$ such that all worlds of size $>N$ and all infinite worlds agree on the truth of
	$\varphi$. If an accessible world of opposite truth is already finite, we are done. Otherwise, some
	accessible world of opposite truth is infinite. Then $W'$ must itself be infinite, since a surjective
	image of a finite set is finite. But an infinite set surjects onto every nonempty finite cardinal,
	so in particular $W'$ accesses some finite world of size $k>N$, and that world has the same truth
	value $\neg t$ as all infinite worlds. Consequently, there exists an accessible finite world of
	opposite truth, and so there is a least integer $n\ge 2$ such that
	\begin{equation*}
		n\models \neg t \text{ for }\varphi \quad\text{(that is\ } n\models \neg\varphi \text{ if } 1\models\varphi,\ \text{and } n\models \varphi \text{ if } 1\models\neg\varphi\text{)}.
	\end{equation*}
	By minimality of $n$, every nonempty world of size $<n$ agrees with $1$ on the truth of
	$\varphi$, and therefore has truth value $t$ for $\varphi$. Since $n$ is accessible from $W'$, we also
	have $|W'|\ge n$ if $W'$ is finite, while if $W'$ is infinite then it trivially accesses every finite
	world. In either case, $W'$ accesses the world $n$.

	At the world $n$, the truth value that holds there (either $\varphi$ or $\neg\varphi$) is
	penultimate: it is true at $n$, it is possible to reach a singleton world $1$ where it becomes
	false, and once it becomes false it stays false at all further surjective images, because every
	surjective image of a nonempty $k$-element world has size $\le k$, and by minimality of $n$ all
	worlds of size $<n$ already agree with $1$ on $\varphi$.

	Thus, at every world $W'$ in the cone above $W$, every contingent sentence $\varphi$ possibly
	attains a penultimate truth-value. Since $\theoryf{S4.3}$ is already valid throughout the cone above
	$W$, and hence in particular $\theoryf{S4}$ is valid there, the \nameref{Lowerboundstheorem2} yields
	that $\theoryf{Grz}$ is valid at every world in the cone above $W$.

	\smallskip

	Finally, applying the \nameref{ConeLemma} to \theoryf{Grz} and \theoryf{S4.3} shows that the
	propositional modal theory \theoryf{Grz.3} is valid at $W$.
\end{proof}

\begin{theorem}\label{Theorem.Surj-infinite-worlds-Grz.3}
    The propositional modal validities of an infinite world in the category $\Sets[$\onto$]$ of sets and surjective functions with respect to sentential substitution instances, constitute precisely the propositional modal theory \theoryf{Grz.3}. More precisely, for every infinite world $W$ in $\Sets[$\onto$]$, and for every first-order language $\mathscr{L}$, modal or not, such that $\Lofeq \subseteq \mathscr{L} \subseteq \LDiamondofeq$, we have
	\begin{equation*}
		\Val_{\Sets[$\onto$]}(W, \mathscr{L}) = \theoryf{Grz.3}.
	\end{equation*}
\end{theorem}

\begin{proof}
    Suppose $W$ is an infinite world. The fact that this world validates \theoryf{Grz.3} for any kinds of substitution instances is just a special case of lemma~\ref{Lemma.Surj-worlds-contain-Grz.3}. Let us therefore show that the validities are contained in \theoryf{Grz.3}.

	Let $N$ be an arbitrarily large natural number.
    Consider the sequence of statements $r_N, r_{N-1}, \ldots, r_1$ where $r_n$ asserts ``there are at most $n$ elements". This is an uncranked ratchet at $W$. One can always surject a larger set onto every smaller non-empty set. It is therefore clear that any given $r_n$ implies $\necessary r_m$ for $m>n$. For the same reason, from any world where the ratchet has a given volume, we can crank it further up to any higher volume, all the way up to $r_1$. And so, by the~\nameref{Upperboundstheorem2}, the propositional modal assertions that are valid at $W$ with respect to sentences are contained in the propositional modal theory \theoryf{Grz.3}.
\end{proof}

\begin{theorem}\label{Theorem.Surj-world-size-n-sentences-Grz.3Jn}
    The propositional modal validities of every world of size $n>0$ in the category $\Sets[$\onto$]$ of sets and surjective functions with respect to sentential substitution instances constitute precisely the propositional modal theory $\theoryf{Grz.3J}_n$. More precisely, for every world $W$ in $\Sets[$\onto$]$ with $|W|=n>0$, and for every first-order language $\mathscr{L}$, modal or not, such that $\Lofeq \subseteq \mathscr{L} \subseteq \LDiamondofeq$, we have
    \begin{equation*}
        \Val_{\Sets[$\onto$]}(W, \mathscr{L})=\theoryf{Grz.3J}_n.
    \end{equation*}
\end{theorem}

\begin{proof}
	Let $W$ be a world in $\Sets[$\onto$]$ with $|W|=n>0$, and fix a first-order language $\mathscr{L}$, modal or not, such that $\Lofeq \subseteq \mathscr{L} \subseteq \LDiamondofeq$.
	Let $F_n=\langle w_0<\cdots < w_{n-1}\rangle$ be the finite linear order of length $n$, where the node $w_i$ is intended to represent the cardinality $n-i$.

	\medskip
	\noindent\textit{Lower bound.}
	Let $\varphi(p_0,\ldots,p_k)$ be any propositional modal assertion in $\theoryf{Grz.3J}_n$, and let $\psi_0,\ldots,\psi_k$ be arbitrary sentences of $\mathscr{L}$.
	By the~\nameref{Modalityeliminationtheorem.Sets}, each $\psi_j$ is equivalent in $\Sets[$\onto$]$ to a finite Boolean combination of exact-cardinality sentences $\sigma_m$.
	Since we are dealing with sentences, no atomic equalities remain, and therefore the truth of each $\psi_j$ depends only on the cardinality of the underlying set.
	We may therefore define a propositional Kripke model $M$ on the frame $F_n$ by stipulating that for each $i<n$ and $j\le k$,
	\begin{align*}
		(M,w_i)\models p_j
		\qquad\text{if and only if}\qquad
		&\text{some (equivalently every) world of size }\\
		&n-i\text{ satisfies }\psi_j.
	\end{align*}
	We claim that for every propositional modal assertion $\chi(p_0,\ldots,p_k)$ and every $i<n$,
	\begin{equation*}
		(M,w_i)\models \chi(p_0,\ldots,p_k)
		\qquad\text{if and only if}\qquad
		U\models \chi(\psi_0,\ldots,\psi_k)
	\end{equation*}
	for any world $U$ of size $n-i$ in $\Sets[$\onto$]$.
	The proof is by induction on $\chi$.
	The atomic and Boolean cases are immediate from the definition of $M$.
	For the modal case, suppose first that $(M,w_i)\models\possible\chi$.
	Then there is some $j\ge i$ with $(M,w_j)\models\chi$.
	If $U$ has size $n-i$, then because $n-j\le n-i$ and both are positive, there is a surjection from $U$ onto a world $V$ of size $n-j$.
	By the induction hypothesis, $V\models \chi(\psi_0,\ldots,\psi_k)$, and hence $U\models \possible\chi(\psi_0,\ldots,\psi_k)$.
	Conversely, if $U\models \possible\chi(\psi_0,\ldots,\psi_k)$, then there is a surjection from $U$ onto a world $V$ of some positive size $m\le n-i$ such that $V\models\chi(\psi_0,\ldots,\psi_k)$.
	Writing $m=n-j$, we have $j\ge i$, and by the induction hypothesis $(M,w_j)\models\chi$.
	Thus $(M,w_i)\models\possible\chi$.
	This completes the induction.

	Since $\varphi\in \theoryf{Grz.3J}_n$, it is valid on every finite linear order of length $n$, in particular on $F_n$.
	Hence $(M,w_0)\models\varphi$, and by the claim above it follows that
	\begin{equation*}
		W\models \varphi(\psi_0,\ldots,\psi_k).
	\end{equation*}
	As the substitution instance was arbitrary, we conclude that
	\begin{equation*}
		\theoryf{Grz.3J}_n \subseteq \Val_{\Sets[$\onto$]}(W,\mathscr{L}).
	\end{equation*}

	\medskip
	\noindent\textit{Upper bound.}
	Suppose now that $\varphi\notin \theoryf{Grz.3J}_n$.
	By the finite frame characterization of $\theoryf{Grz.3J}_n$, there is a propositional Kripke model $M$ based on a finite linear order
	\begin{equation*}
		F=\langle v_0<\cdots < v_{\ell-1}\rangle
	\end{equation*}
	of length $\ell\le n$, such that $(M,v_0)\models \neg\varphi$.

	We now label $F$ for the world $W$ using cardinality sentences.
	For $0\le i<\ell-1$, let
	\begin{equation*}
		\Phi_{v_i}:=\sigma_{n-i},
	\end{equation*}
	and let
	\begin{equation*}
		\Phi_{v_{\ell-1}}:=\bigvee_{m=1}^{n-\ell+1}\sigma_m.
	\end{equation*}
	Since $|W|=n$, we have $W\models \Phi_{v_0}$.
	Every world in the cone above $W$ has size between $1$ and $n$, and so satisfies exactly one of the labels $\Phi_{v_i}$.
	Moreover, if a world satisfies $\Phi_{v_i}$, then the possible cardinalities of its surjective images are exactly those corresponding to the labels $\Phi_{v_j}$ with $j\ge i$.
	Thus $v_i\le_F v_j$ if and only if every world satisfying $\Phi_{v_i}$ accesses a world satisfying $\Phi_{v_j}$.
	So $v\mapsto \Phi_v$ is a labeling of $F$ for $W$.

	By the \nameref{LabelingLemma}, there is an assignment of the propositional variables to $\Lofeq$-sentences such that
	\begin{equation*}
		(M,v_0)\models \chi(p_0,\ldots,p_k)
		\qquad\text{if and only if}\qquad
		W\models \chi(\psi_{p_0},\ldots,\psi_{p_k})
	\end{equation*}
	for every propositional modal assertion $\chi$.
	Since $\Lofeq \subseteq \mathscr{L}$, these witnessing sentences also belong to $\mathscr{L}$.
	Applying this to $\chi=\varphi$ and using $(M,v_0)\models\neg\varphi$, we obtain a sentential substitution instance in $\mathscr{L}$ witnessing that $\varphi$ is not valid at $W$.
	Therefore
	\begin{equation*}
		\Val_{\Sets[$\onto$]}(W,\mathscr{L})\subseteq \theoryf{Grz.3J}_n.
	\end{equation*}

	Combining the two inclusions yields
	\begin{equation*}
		\Val_{\Sets[$\onto$]}(W,\mathscr{L})=\theoryf{Grz.3J}_n,
	\end{equation*}
	as claimed.
\end{proof}

\begin{lemma}\label{Lemma.Surj-world-parameters-contain-Grz.2}
	The propositional modal theory \theoryf{Grz.2} is valid at every world in the category $\Sets[$\onto$]$ of sets and surjective functions, for any assertions whatsoever, even formulaic assertions, with parameters, in the modal language of equality $\LDiamondofeq$. In other words, for any world $W$ in $\Sets[$\onto$]$, we have
	\begin{equation*}
		\theoryf{Grz.2} \of \Val_{\Sets[$\onto$]}(W, {\LDiamondofeq}_{W}).
	\end{equation*}
\end{lemma}

\begin{proof}
	If $W=\emptyset$, then the cone above $W$ is trivial and hence $\theoryf{Triv}$ (and therefore $\theoryf{Grz.2}$) is valid at $W$ for any substitution instances. So assume that $W$ is nonempty. In this case every world in the cone above $W$ is nonempty as well, since there is no surjection from a nonempty set onto~$\emptyset$.

	\medskip
	\noindent\textit{(1) Validity of $\theoryf{S4.2}$ on the cone.}
	Let $W'$ be a world in the cone above $W$. Consider any span in $\Sets[$\onto$]$,
	\begin{equation*}
		U_0 \xleftarrow{f_0} W' \xrightarrow{f_1} U_1.
	\end{equation*}
	Since $U_0$ and $U_1$ are nonempty, the unique maps $g_0 \colon U_0 \to \singleton{s}$ and $g_1 \colon U_1 \to \singleton{s}$ are surjections and map everything to a common singleton world, thereby completing the span and forming a commutative square. Consequently, the cone above $W$ is amalgamable. By the~\nameref{Lowerboundstheorem}, the propositional modal theory $\theoryf{S4.2}$ is valid at $W$ with respect to all substitution instances in ${\LDiamondofeq}_{W}$.

	\medskip
	\noindent\textit{(2) Penultimacy on the cone.}
	Let $W'$ be any world in the cone above $W$, let $\bar a=(a_0,\dots,a_{m-1})$ be a tuple from $W'$, and let $\varphi(\bar x)$ be any formula of ${\LDiamondofeq}_{W'}$. Assume that $\varphi[\bar a]$ is contingent at $W'$.

	Work in the expanded language $\Lofeq(\bar c)$ with constants $\bar c=(c_0,\dots,c_{m-1})$. For any surjection $f\colon W'\twoheadrightarrow U$ in the cone above $W'$, let $U^f$ denote the expansion of $U$ interpreting each $c_i$ as $f(a_i)$; thus $U\models \psi[f(\bar a)]$ is equivalent to $U^f\models \psi(\bar c)$ for any $\Lofeq$-formula $\psi$.

	By the~\nameref{Modalityeliminationtheorem.Sets} (applied with parameters), there is a non-modal $\Lofeq(\bar c)$-sentence $\theta(\bar c)$, which is a finite Boolean combination of atomic equalities/inequalities among the $c_i$ and sentences $\sigma_n$ asserting that there are exactly $n$ elements, such that
	\begin{equation*}
		\Sets[$\onto$]\models \forall \bar x\,\bigl(\varphi(\bar x)\leftrightarrow \theta(\bar x)\bigr).
	\end{equation*}
	In particular, for every surjection $f\colon W'\twoheadrightarrow U$ we have
	\begin{equation*}
		W'\models \varphi[\bar a]\quad\text{iff}\quad W'\models \theta[\bar a],
		\qquad\text{and}\qquad
		U\models \varphi[f(\bar a)]\quad\text{iff}\quad U^f\models \theta(\bar c).
	\end{equation*}

		Let
		\begin{equation*}
			N=\max\bigl(\{0\}\cup\{n\mid \sigma_n\text{ occurs in }\theta\}\bigr),
		\end{equation*}
		so that $N$ is defined even if no exact-cardinality sentence occurs in $\theta$. For each surjection $f\colon W'\twoheadrightarrow U$, define the \emph{$\theta$-rank} of the expansion $U^f$ by
	\begin{equation*}
		\mathrm{rk}(U^f)=(s(U),P_f),
	\end{equation*}
	where $s(U)=\min(|U|,N+1)$ and $P_f$ is the partition of $\{0,\dots,m-1\}$ induced by equality among the constants in $U^f$. Partially order such ranks by declaring
	\begin{equation*}
		(s,P)\preceq (s',P')\quad\text{iff}\quad s\le s'\ \text{and}\ P\ \text{is coarser than}\ P'.
	\end{equation*}
	If $g\colon U\twoheadrightarrow V$ is a surjection, then $\mathrm{rk}(V^{g\compose f})\preceq \mathrm{rk}(U^f)$, since size can only decrease and equalities among the images of $\bar a$ can only increase.

	Let $\pi\colon W'\twoheadrightarrow \singleton{s}$ be the unique surjection to a singleton, and consider the bottom expansion $(\singleton{s})^{\pi}$, in which all constants $c_i$ are equal. Let $t$ be the truth value of $\theta(\bar c)$ there. Since $\varphi[\bar a]$ is contingent at $W'$, there exists some surjection $f\colon W'\twoheadrightarrow U$ such that $U^f\models \theta(\bar c)$ has truth value $\neg t$. Choose such an $f$ so that $\mathrm{rk}(U^f)$ is $\preceq$-minimal among all expansions reachable from $W'$ on which $\theta(\bar c)$ has truth value $\neg t$.

	If $t$ is true in $(\singleton{s})^{\pi}$, then $U^f\models \neg\theta(\bar c)$, so $\neg\theta(\bar c)$ is true at $U^f$, and $\theta(\bar c)$ is possible at $U^f$ (witnessed by the surjection $U\twoheadrightarrow \singleton{s}$). Moreover, if $g\colon U\twoheadrightarrow V$ and $V^{g\compose f}\models \theta(\bar c)$, then $\mathrm{rk}(V^{g\compose f})\prec \mathrm{rk}(U^f)$; by minimality of $\mathrm{rk}(U^f)$ there is no world of smaller rank where $\neg\theta$ holds. Hence every further surjective image of $V^{g\compose f}$ also satisfies $\theta(\bar c)$, and so $V^{g\compose f}\models \necessary\theta(\bar c)$. This shows that $\neg\theta(\bar c)$ is penultimate at $U^f$.

	If $t$ is false in $(\singleton{s})^{\pi}$, the same argument with $\theta$ and $\neg\theta$ interchanged shows that $\theta(\bar c)$ is penultimate at $U^f$.

	Therefore every contingency at every world in the cone above $W$ possibly attains a penultimate truth-value. By the~\nameref{Lowerboundstheorem2}, it follows that $\theoryf{Grz}$ is valid at $W$ with respect to ${\LDiamondofeq}_{W}$. Together with the validity of $\theoryf{S4.2}$ established in (1), this yields validity of $\theoryf{Grz.2}$ at $W$.
\end{proof}

\begin{theorem}\label{Theorem.Surj-infinite-worlds-formulas-Grz.2}
    The propositional modal validities of an infinite world in the category $\Sets[$\onto$]$ of sets and surjective functions with respect to formulaic substitution instances, with parameters, constitute precisely the propositional modal theory \theoryf{Grz.2}. More precisely, for every infinite world $W$ in $\Sets[$\onto$]$, and for every first-order language $\mathscr{L}$, modal or not, such that $\Lofeq \subseteq \mathscr{L} \subseteq \LDiamondofeq$, we have
	\begin{equation*}
		\Val_{\Sets[$\onto$]}(W, \mathscr{L}_{W}) = \theoryf{Grz.2}.
	\end{equation*}
\end{theorem}

\begin{proof}
    Suppose $W$ is an infinite world. The fact that this world validates \theoryf{Grz.2} is just a special case of lemma~\ref{Lemma.Surj-world-parameters-contain-Grz.2}. Let us therefore show that the validities are contained in \theoryf{Grz.2}. By lemma~\ref{Lemma.Surj-world-parameters-contain-Grz.2} the propositional modal logic \theoryf{S4.2} is valid at $W$. Observation~\ref{Observation.Under-S4.2-weak-buttons-are-buttons} states that if \theoryf{S4.2} is valid, then every weak button is a button. For any two distinct individuals there is a surjection to an accessible quotient world identifying them, and once merged the individuals remain equal in all further accessible worlds.

    By the~\nameref{Upperboundstheorem2}, it therefore suffices to exhibit arbitrarily large finite families
    of independent buttons at $W$. Fix $N\in\mathbb{N}$ and choose pairwise distinct parameters
    $u_1,v_1,\ldots,u_N,v_N\in W$.

    Let $\rho(\bar u,\bar v)$ be the assertion saying that some cross-wiring occurs among these
    parameters, namely:
    \begin{equation*}
        \bigvee_{j\neq k}\Bigl(u_j=u_k  \lor u_j=v_k \lor v_j=v_k\Bigr).
    \end{equation*}
    For each $1\le i\le N$, define $b_i$ as $(u_i=v_i) \lor \rho(\bar u,\bar v)$.
    At $W$ we have $u_i\neq v_i$ and $\neg\rho(\bar u,\bar v)$ (all parameters are distinct),
    hence each $b_i$ is unpushed at $W$.

    Each $b_i$ is a button. Indeed, from any world $W'$ above $W$, if $b_i$ is not yet true at $W'$
    then in particular $\rho(\bar u,\bar v)$ is false at $W'$, and the images of $u_i$ and $v_i$
    in $W'$ are distinct. Let $W''$ be the quotient of $W'$ obtained by identifying these two points
    and no others, and let $\pi\colon W'\twoheadrightarrow W''$ be the quotient map. Then $\pi$ is a
    surjection, so $W''$ is accessible from $W'$ in $\Sets[$\onto$]$, and at $W''$ we have $u_i=v_i$,
    hence $b_i$ holds. Moreover, once $u_i=v_i$ holds at some world it remains true at all further
    accessible worlds because surjections preserve equality; likewise, once $\rho(\bar u,\bar v)$
    holds it remains true for the same reason. Thus $b_i$ becomes necessary once it becomes true.

    To see independence, fix any $S\subseteq\{1,\ldots,N\}$. Let $\sim_S$ be the equivalence relation
    on $W$ generated by $u_i\sim_S v_i$ for $i\in S$ and no other identifications, and set
    $W_S=W/{\sim_S}$. Let $q_S\colon W\twoheadrightarrow W_S$ be the quotient map. Then $q_S$
    is a surjection, so $W_S$ is accessible from $W$ in $\Sets[$\onto$]$. By construction,
    $\rho(\bar u,\bar v)$ is still false at $W_S$, and $u_i=v_i$ holds at $W_S$ if and only if
    $i\in S$. Consequently, exactly the buttons $b_i$ with $i\in S$ are pushed at $W_S$, and the
    others remain unpushed.

    Since $N$ was arbitrary, $W$ admits arbitrarily large finite families of independent buttons,
    and so, by the~\nameref{Upperboundstheorem2}, the propositional modal validities of $W$ are
    contained in the theory $\theoryf{Grz.2}$.
\end{proof}

If one restricts instead to the full subcategory $\InfSets[$\onto$]$ of infinite sets and surjections, the sentential validities again collapse to \theoryf{Triv} by corollary~\ref{Corollary.InfSets-sentences-Triv}. The formulaic validities remain \theoryf{Grz.2} as well. Since passing from $\Sets[$\onto$]$ to $\InfSets[$\onto$]$ removes the finite quotient worlds, the lower bound is not an automatic consequence of lemma~\ref{Lemma.Surj-world-parameters-contain-Grz.2}. We therefore record a direct proof in the infinite-only subcategory: once finite cardinalities disappear, every formula factors through the finite partition lattice generated by the named parameters.

\begin{theorem}\label{Theorem.InfSurj-infinite-worlds-formulas-Grz.2}
    The propositional modal validities of a world in the category $\InfSets[$\onto$]$ of infinite sets and surjective functions with respect to formulaic substitution instances, with parameters, constitute precisely the propositional modal theory \theoryf{Grz.2}. More precisely, for every world $W$ in $\InfSets[$\onto$]$, and for every first-order language $\mathscr{L}$, modal or not, such that $\Lofeq \subseteq \mathscr{L} \subseteq \LDiamondofeq$, we have
    \begin{equation*}
        \Val_{\InfSets[$\onto$]}(W,\mathscr{L}_{W})=\theoryf{Grz.2}.
    \end{equation*}
\end{theorem}

\begin{proof}
    Let $W$ be a world in $\InfSets[$\onto$]$.

    We first prove the upper bound. By statement~(2) of the~\nameref{Upperboundstheorem2}, it suffices to show that $W$ has arbitrarily large finite independent families of unpushed buttons. Fix $N\in\omega$, and choose pairwise distinct parameters
    \begin{equation*}
        u_1,v_1,\ldots,u_N,v_N\in W.
    \end{equation*}
    Let $\rho(\bar u,\bar v)$ be the assertion that some unwanted cross-identification occurs among these parameters, namely
    \begin{equation*}
        \bigvee_{j\neq k}\Bigl(u_j=u_k \lor u_j=v_k \lor v_j=v_k\Bigr).
    \end{equation*}
    For each $1\leq i\leq N$, define
    \begin{equation*}
        b_i := (u_i=v_i)\lor \rho(\bar u,\bar v).
    \end{equation*}
    At $W$ all the displayed parameters are distinct, so each $b_i$ is unpushed.

    Each $b_i$ is a button at $W$. Indeed, let $W'$ be any world in the cone above $W$ at which $b_i$ is false. Then $\rho(\bar u,\bar v)$ is false at $W'$, and the images of $u_i$ and $v_i$ in $W'$ are distinct. Form the quotient of $W'$ obtained by identifying just those two points and no others. Since only finitely many identifications are made, the quotient remains infinite, and the quotient map is a surjection in $\InfSets[$\onto$]$. In the quotient we have $u_i=v_i$, so $b_i$ becomes true. Moreover, once $u_i=v_i$ or $\rho(\bar u,\bar v)$ holds at some world, it remains true at all further accessible worlds, since surjections preserve equality. Thus $b_i$ is necessarily possibly necessary.

    To see independence, fix any $S\subseteq\{1,\ldots,N\}$. Let $\sim_S$ be the equivalence relation on $W$ generated by $u_i\sim_S v_i$ for $i\in S$ and no other identifications, and let
    \begin{equation*}
        q_S\colon W\twoheadrightarrow W/{\sim_S}
    \end{equation*}
    be the quotient map. Since only finitely many pairs are identified, the quotient $W/{\sim_S}$ is still infinite, and hence is a world of $\InfSets[$\onto$]$. By construction, $\rho(\bar u,\bar v)$ is false in $W/{\sim_S}$, and $u_i=v_i$ holds there if and only if $i\in S$. Consequently, exactly the buttons $b_i$ with $i\in S$ are pushed at that world, and the remaining ones are unpushed. Since $N$ was arbitrary, $W$ has arbitrarily large finite independent families of unpushed buttons. Therefore
    \begin{equation*}
        \Val_{\InfSets[$\onto$]}(W,\mathscr{L}_{W})\subseteq \theoryf{Grz.2}.
    \end{equation*}

    We now prove the lower bound. Let
    \begin{equation*}
        \sigma(p_0,\ldots,p_k)\in \theoryf{Grz.2}
    \end{equation*}
    be any propositional modal assertion, and let
    \begin{equation*}
        \sigma\bigl(\psi_0,\ldots,\psi_k\bigr)
    \end{equation*}
    be any substitution instance arising from assertions $\psi_0,\ldots,\psi_k$ in $\mathscr{L}_{W}$. Since only finitely many parameters occur in the $\psi_i$, there is a finite tuple
    \begin{equation*}
        \bar a=(a_0,\ldots,a_{n-1})
    \end{equation*}
    of distinct elements of $W$ containing all of them.

    By corollary~\ref{Corollary.InfSets-formulas-partitions}, each $\psi_i(\bar a)$ is equivalent in $\InfSets[$\onto$]$ to a finite disjunction of partition formulas on $\bar a$. Let $F_n$ be the partition lattice of $n$, ordered by refinement, and for each partition $P\in F_n$ let $\Phi_P(\bar a)$ be the corresponding equality pattern on the tuple $\bar a$. Every world in the cone above $W$ satisfies exactly one $\Phi_P(\bar a)$, and every $P\in F_n$ is realized above $W$ by identifying only the finitely many parameter-values required by $P$. Moreover, for partitions $P,Q\in F_n$, a world satisfying $\Phi_P(\bar a)$ can access a world satisfying $\Phi_Q(\bar a)$ if and only if $Q$ is a coarsening of $P$, that is, if and only if $P\leq_{F_n}Q$. The forward implication holds because any surjection can only identify parameter-values, never separate them. For the converse, if $Q$ coarsens $P$, then by identifying only those parameter-values required to pass from $P$ to $Q$ and leaving all other elements untouched, we obtain a quotient world that is still infinite and that satisfies $\Phi_Q(\bar a)$. Thus $P\mapsto \Phi_P(\bar a)$ is a labeling of $F_n$ for $W$.

    For each $i\leq k$, let
    \begin{equation*}
        \begin{aligned}
            X_i:=\{P\in F_n\mid
            &\text{every world in the cone above $W$ satisfying }\\
            &\Phi_P(\bar a)\text{ also satisfies }\psi_i(\bar a)\}.
        \end{aligned}
    \end{equation*}
    Since each $\psi_i(\bar a)$ depends only on the partition realized by $\bar a$, the set $X_i$ is well defined; equivalently, throughout the cone above $W$ we have
    \begin{equation*}
        \psi_i(\bar a)\iff \bigvee_{P\in X_i}\Phi_P(\bar a).
    \end{equation*}
    Define a propositional Kripke model $M$ on the frame $F_n$ by stipulating, for each $i\leq k$ and each $P\in F_n$,
    \begin{equation*}
        (M,P)\models p_i
        \qquad\text{if and only if}\qquad
        P\in X_i.
    \end{equation*}
    We claim that for every propositional modal formula $\theta(p_0,\ldots,p_k)$, every partition $P\in F_n$, and every world $V$ in the cone above $W$ satisfying $\Phi_P(\bar a)$,
    \begin{equation*}
        V\models \theta(\psi_0,\ldots,\psi_k)
        \qquad\text{if and only if}\qquad
        (M,P)\models \theta(p_0,\ldots,p_k).
    \end{equation*}
    The proof is by induction on the complexity of $\theta$. The atomic case is exactly the definition of the sets $X_i$, and the Boolean connectives are immediate. For the modal step, let $\theta=\possible\eta$. If $V\models \possible\eta(\psi_0,\ldots,\psi_k)$, choose an accessible world $V'$ with $V'\models \eta(\psi_0,\ldots,\psi_k)$, and let $Q$ be the partition realized there by the image of $\bar a$. Then $Q$ coarsens $P$, so $P\leq_{F_n}Q$; by the induction hypothesis, $(M,Q)\models \eta(p_0,\ldots,p_k)$, and hence $(M,P)\models \possible\eta(p_0,\ldots,p_k)$. Conversely, if $(M,P)\models \possible\eta(p_0,\ldots,p_k)$, choose $Q\in F_n$ with $P\leq_{F_n}Q$ and $(M,Q)\models \eta(p_0,\ldots,p_k)$. By identifying in $V$ only the finitely many parameter-values needed to pass from $P$ to $Q$, we obtain a quotient world $V_Q$ in $\InfSets[$\onto$]$ such that $V_Q\models \Phi_Q(\bar a)$. By the induction hypothesis, $V_Q\models \eta(\psi_0,\ldots,\psi_k)$, so $V\models \possible\eta(\psi_0,\ldots,\psi_k)$. This completes the induction.

    Since $F_n$ is a finite partition lattice, it is in particular a finite directed partial order. By the finite-frame characterization of \theoryf{Grz.2}, every assertion in \theoryf{Grz.2} is valid on $F_n$. Hence
    \begin{equation*}
        (M,P_0)\models \sigma(p_0,\ldots,p_k),
    \end{equation*}
    where $P_0$ is the discrete partition realized by the tuple $\bar a$ in the ground world $W$. Applying the displayed equivalence with $V=W$, it follows that
    \begin{equation*}
        W\models \sigma\bigl(\psi_0,\ldots,\psi_k\bigr).
    \end{equation*}
    Since $\sigma$ and the substitution instance were arbitrary, we conclude that
    \begin{equation*}
        \theoryf{Grz.2}\subseteq \Val_{\InfSets[$\onto$]}(W,\mathscr{L}_{W}).
    \end{equation*}

    Combining the lower and upper bounds yields
    \begin{equation*}
        \Val_{\InfSets[$\onto$]}(W,\mathscr{L}_{W})=\theoryf{Grz.2},
    \end{equation*}
    as desired.
\end{proof}

The propositional modal theory $\theoryf{Partition}_n$ is the smallest normal propositional modal theory valid in the partition lattice of an $n$-element set, ordered by refinement. Since this lattice is finite for fixed $n$, no separate finite-model-property argument is needed here.

\begin{theorem}\label{Theorem.Surj-world-size-n-formulas-Partition}
    The propositional modal validities of a world of size $n$ in the category $\Sets[$\onto$]$ of sets and surjective functions with respect to formulaic substitution instances, with parameters, constitute precisely the propositional modal theory $\theoryf{Partition}_n$. More precisely, for every $W_0$ in $\Sets[$\onto$]$ such that $|W_0| = n$, and for every first-order language $\mathscr{L}$, modal or not, such that $\Lofeq \subseteq \mathscr{L} \subseteq \LDiamondofeq$, we have
	\begin{equation*}
		\Val_{\Sets[$\onto$]}(W_0, {\mathscr{L}}_{W_0}) = \theoryf{Partition}_n.
	\end{equation*}
\end{theorem}

\begin{proof}
	By observation~\ref{Observation.Surj-and-FinSurj-same}, it is sufficient to work in the category $\FinSets[$\onto$]$.

	Let $W_0$ be a world of size $n$ in $\FinSets[$\onto$]$ and expand it to the language $\Lofeq(a_i\mid i<n)$ by interpreting the constants $a_i$ as a fixed enumeration of the elements of $W_0$; call this expansion $W_0^{+}$.
	Let $F$ be the partition lattice of $n$ ordered by refinement (as in section~\ref{Section.Modality-elimination-in-Sets}), and let $w_0$ be its bottom element (the discrete partition into singletons).

	\medskip
	\noindent\textit{(1) $\theoryf{Partition}_n$ is valid at $W_0$ for formulaic substitution instances (with parameters), even in $\LDiamondofeq$.}

	Build a propositional Kripke model $M$ on the frame $F$ as follows. For each atomic equality $a_i=a_j$ associate a propositional variable $p_{ij}$, and stipulate that
	\begin{align*}
		(M,w)\models p_{ij}
		\qquad\text{if and only if}\qquad
		&\text{$i$ and $j$ belong to the same block}\\
		&\text{of the partition }w.
	\end{align*}

	For every surjection $f\colon W_0\to W$, let $W^f$ denote the induced expansion of $W$ to the language $\Lofeq(a_i\mid i<n)$ in which each $a_i$ is interpreted as $f(a_i)$, and let $w_f\in F$ be the kernel-partition of $f$ (so $i$ and $j$ are in the same block of $w_f$ iff $f(a_i)=f(a_j)$). Then for all $i,j<n$,
	\begin{equation*}
		(M,w_f)\models p_{ij}\qquad\text{if and only if}\qquad W^f\models a_i=a_j.
	\end{equation*}

	Moreover, if $f\colon W_0\to W$ and $g\colon W_0\to U$ are surjections, then $w_f\leq_F w_g$ (that is\ $w_f$ refines $w_g$) if and only if there exists a surjection $h\colon W\to U$ with $g=h\compose f$. In that case, $h$ is exactly the accessibility mapping witnessing that $W^f$ accesses $U^g$ in the cone above $W_0$.

	Define a translation $\psi\mapsto\psi^*$ from $\LDiamondofeq(a_i\mid i<n)$-formulas to propositional modal formulas in the variables $p_{ij}$ by
	\begin{align*}
		(a_i=a_j)^*&=p_{ij}, & (\neg\psi)^*&=\neg\psi^*,\\
		(\psi_0\wedge\psi_1)^*&=\psi_0^*\wedge\psi_1^*, & (\possible\psi)^*&=\possible\psi^*.
	\end{align*}
	and (since every element of any surjective image of $W_0$ is named by some $a_i$)
	\begin{equation*}
		(\exists x\,\psi(x))^*=\bigvee_{i<n}(\psi(a_i))^*.
	\end{equation*}
	A routine induction on $\psi$ shows that for every surjection $f\colon W_0\to W$,
	\begin{equation*}
		(M,w_f)\models \psi^* \qquad \text{if and only if} \qquad W^f\models \psi.
	\end{equation*}

	Now suppose $\varphi$ is not valid at $W_0$ with respect to substitution instances from $\LDiamondofeq$ (with parameters from $W_0$). Then for some substitution $p\mapsto \psi_p$ we have $W_0^{+}\not\models \varphi(\psi_{p_0},\ldots,\psi_{p_k})$. Translating each $\psi_{p_i}$ to $\psi_{p_i}^*$ and using the displayed equivalence at $w_0$, we obtain
	\begin{equation*}
		(M,w_0)\not\models \varphi(\psi_{p_0}^*,\ldots,\psi_{p_k}^*).
	\end{equation*}
	Since $\theoryf{Partition}_n$ is a normal theory (hence closed under substitutions), this shows $\varphi\notin \theoryf{Partition}_n$. Equivalently, $\theoryf{Partition}_n\subseteq \Val_{\FinSets[$\onto$]}(W_0,{\LDiamondofeq}_{W_0})$.

	\medskip
	\noindent\textit{(2) The validities are contained in $\theoryf{Partition}_n$.}

	For each partition $P\in F$, let $\Phi_P$ be the $\Lofeq(a_i\mid i<n)$-sentence $P[\bar a]$ expressing that the equalities among the constants $a_i$ are exactly those prescribed by $P$ (cf.\ section~\ref{Section.Modality-elimination-in-Sets}). Then $W_0^{+}\models \Phi_{w_0}$, and every world in the cone above $W_0$ satisfies exactly one $\Phi_P$.

	Finally, for partitions $P,Q\in F$, we have $P\leq_F Q$ if and only if $P$ refines $Q$ (equivalently, $Q$ is coarser than $P$). Viewing partitions as formulas, this corresponds exactly to the accessibility relation in the cone: whenever $P\leq_F Q$, every world satisfying $\Phi_P$ can be further mapped (via a surjection) to a world satisfying $\Phi_Q$, so $\Phi_P\rightarrow \possible \Phi_Q$ holds throughout the cone. Thus $P\mapsto \Phi_P$ is a labeling of $F$ for $W_0$. By the~\nameref{LabelingLemma}, it follows that
	\begin{equation*}
		\Val_{\FinSets[$\onto$]}(W_0,{\Lofeq}_{W_0})\subseteq \theoryf{Partition}_n.
	\end{equation*}

	Now for any $\mathscr L$ with $\Lofeq\subseteq \mathscr L\subseteq \LDiamondofeq$, monotonicity of validities with respect to the language of substitution instances yields
	\begin{align*}
		\theoryf{Partition}_n
		&\subseteq \Val_{\FinSets[$\onto$]}(W_0,{\LDiamondofeq}_{W_0})
		\subseteq \Val_{\FinSets[$\onto$]}(W_0,{\mathscr L}_{W_0})\\
		&\subseteq \Val_{\FinSets[$\onto$]}(W_0,{\Lofeq}_{W_0})
		\subseteq \theoryf{Partition}_n.
	\end{align*}
	and hence $\Val_{\FinSets[$\onto$]}(W_0,{\mathscr L}_{W_0})=\theoryf{Partition}_n$. The same equality holds in $\Sets[$\onto$]$ by observation~\ref{Observation.Surj-and-FinSurj-same}.
\end{proof}

\section{The modal theory of the category of sets and injective functions or inclusions}\label{Section.category-of-sets-with-injections}

\begin{observation}\label{Observation.Incl-and-Inj-same}
	A world in the category of sets and injective functions has precisely the same modal truths as that same world considered in the category of sets and inclusions. More precisely, if $W$ is a world, then
	\begin{equation*}
		W \models_{\Sets[$\into$]} \varphi[\bar x] \qquad \text{if and only if} \qquad W \models_{\Sets[$\incl$]} \varphi[\bar x],
	\end{equation*}
	for any assertion $\varphi$ in the modal language of equality $\LDiamondofeq$, with parameters or not. In particular, any world in the category $\Sets[$\into$]$ has the same propositional modal validities as it would have in the category $\Sets[$\incl$]$.
\end{observation}

\begin{observation}\label{Observation.Inj-and-FinInj-same}
A finite world in the category of sets and injective functions or the category of sets and inclusions has precisely the same modal truths as that same world considered in the category of finite sets and injective functions or the category of finite sets and inclusions. More precisely, if $W$ is a finite world, then (1) $W \models_{\Sets[$\into$]} \varphi[\bar x]$, (2) $W \models_{\FinSets[$\into$]} \varphi[\bar x]$, (3) $W \models_{\Sets[$\incl$]} \varphi[\bar x]$, and (4) $W \models_{\FinSets[$\incl$]} \varphi[\bar x]$ are equivalent for any assertion $\varphi$ in the modal language of equality $\LDiamondofeq$, with parameters or not. In particular, any world in the category $\FinSets[$\into$]$ has the same propositional modal validities as it would have in the categories $\Sets[$\into$]$, $\Sets[$\incl$]$, and $\FinSets[$\incl$]$.
\end{observation}

\begin{proof}
	By observation~\ref{Observation.Incl-and-Inj-same}, it suffices to compare the inclusion categories. The inclusion case of the~\nameref{Modalityeliminationtheorem.Sets} gives the same finite Boolean reductions in $\Sets[$\incl$]$ and in $\FinSets[$\incl$]$ when the starting world is finite: inclusions preserve equality patterns, can enlarge finite worlds to any larger finite cardinality, and cannot shrink cardinality. Therefore every $\LDiamondofeq$-assertion is equivalent, in both inclusion categories at the finite world $W$, to the same finite Boolean combination of partition formulas and exact-cardinality sentences. These have the same truth values at $W$ in the unrestricted and finite subcategories, proving the equivalence of (3) and (4). The equivalence with (1) and (2) follows from observation~\ref{Observation.Incl-and-Inj-same} and its finite-subcategory version, which is obtained by the same renaming argument.
\end{proof}


\begin{theorem}\label{Theorem.Inj-infinite-world-parameters-Triv}
    The propositional modal validities of an infinite world in the category $\Sets[$\into$]$ of sets and injective functions, with respect to any assertions whatsoever, constitute precisely the trivial propositional modal theory \theoryf{Triv}. More precisely, for every infinite world $W$ in $\Sets[$\into$]$, and for every first-order language $\mathscr{L}$, modal or not, such that $\Lofeq \of \mathscr{L} \of {\LDiamondofeq}_{W}$, we have
    \begin{equation*}
        \Val_{\Sets[$\into$]}(W, \mathscr{L}) = \theoryf{Triv}.
    \end{equation*}
\end{theorem}

\begin{proof}
	The validities of any Kripke category, with respect to any substitution instances, are always contained in the trivial propositional modal theory \theoryf{Triv}.

    From theorem~\ref{Theorem.world-complete-iff-modalities-trivialize}, to show that \theoryf{Triv} is valid at infinite worlds, it is sufficient to show that the cones above the infinite worlds are model complete. We are in the case of $\Sets[$\into$]$, so this is just the traditional notion of model completeness for the first-order theory of infinite sets in the language of equality. Two models of this theory are isomorphic just in case they have the same cardinality. Moreover, this theory is model complete. Indeed, if $f\colon A\hookrightarrow B$ is an injection between infinite sets, then $f$ is an $\Lofeq$-elementary embedding. Indeed, by the Tarski--Vaught test, any existential condition over parameters from $A$ that is realized in $B$ is witnessed either by one of those parameters or by an element distinct from finitely many parameters; since $A$ is infinite, such a witness can always be chosen from $A$, and hence the embedding is elementary.
\end{proof}

\begin{theorem}\label{Theorem.Inj-finite-worlds-Grz.3}
    The propositional modal validities of a finite world in the category $\Sets[$\into$]$ of sets and injective functions constitute precisely the propositional modal theory \theoryf{Grz.3}. More precisely, for every finite world $W$ in $\Sets[$\into$]$ and every first-order language $\mathscr{L}$, modal or not, such that $\Lofeq \of \mathscr{L} \of {\LDiamondofeq}_{W}$, we have that
    \begin{equation*}
		\Val_{\Sets[$\into$]}(W, \mathscr{L}) = \theoryf{Grz.3}.
	\end{equation*}
\end{theorem}

\begin{proof}
    By monotonicity of validities with respect to the language of substitution instances, it suffices
    to prove the lower bound in the largest language ${\LDiamondofeq}_{W}$. So let us show that
    $\theoryf{Grz.3}$ is valid at the finite world $W$ with respect to assertions in
    ${\LDiamondofeq}_{W}$.

    Fix an arbitrary world $W'$ in the cone above $W$. Since accessibility is by injections, we may
    identify $W$ with its image inside $W'$, and likewise inside every further accessible world.
    By the~\nameref{Modalityeliminationtheorem.Sets}, every assertion in ${\LDiamondofeq}_{W}$ is
    equivalent in $\Sets[$\into$]$ to a finite Boolean combination of atomic equalities and
    inequalities among the named elements of $W$ together with assertions saying that there are
    exactly $k$ further elements outside the distinguished copy of $W$, for finitely many values of
    $k$. Consequently, for any accessible world, the truth of such an assertion depends only on the
    number of elements outside the distinguished copy of $W$: a natural number $r$, or the value
    $\infty$ if there are infinitely many such elements.

    We first show that $\theoryf{S4.3}$ is valid at every world in the cone above $W$ with respect
    to ${\LDiamondofeq}_{W}$. Let $b_0$ and $b_1$ be assertions in ${\LDiamondofeq}_{W}$, and
    suppose towards a contradiction that they are independent weak buttons at $W'$. Then there are
    accessible worlds $U_0,U_1$ from $W'$ with
    \begin{equation*}
        U_0\models \necessary b_0\land \neg b_1
        \qquad\text{and}\qquad
        U_1\models \necessary b_1\land \neg b_0.
    \end{equation*}
    Let $r_0$ and $r_1$ be the numbers of elements outside the distinguished copy of $W$ in $U_0$
    and $U_1$, respectively (either a natural number or $\infty$). These values are linearly
    ordered by the usual order on $\omega\cup\{\infty\}$. Without loss of generality, assume
    $r_0\le r_1$. Then there is an accessible world $V$ from $U_0$ having exactly $r_1$ elements
    outside the distinguished copy of $W$. Since truth of assertions from ${\LDiamondofeq}_{W}$
    depends only on this number, $V$ satisfies exactly the same such assertions as $U_1$, and in
    particular $V\models \neg b_0$. But $U_0\models \necessary b_0$, contradiction. Hence there are
    no independent weak buttons at $W'$. Since $W'$ was arbitrary, lemma~\ref{Lemma.S4.3-valid-iff-no-independent-buttons}
    yields that $\theoryf{S4.3}$ is valid at every world in the cone above $W$.

    Next, let us verify the validity of $\theoryf{Grz}$ on the cone above $W$.
    Fix an arbitrary world $W'$ in that cone, and let $\varphi$ be an assertion in
    ${\LDiamondofeq}_{W}$ that is contingent at $W'$. As observed above, the truth of $\varphi$ at
    accessible worlds depends only on the number of elements outside the distinguished copy of $W$.
    Let $t_\infty$ be the common truth value of $\varphi$ at all worlds with infinitely many such
    extra elements. Because $\varphi$ is equivalent to a finite Boolean combination of assertions of
    the form ``there are exactly $k$ further elements'', there exists $N$ such that every world with
    more than $N$ further elements, and every world with infinitely many further elements, has truth
    value $t_\infty$ for $\varphi$.

    If $W'$ already has infinitely many further elements, then every world accessible from $W'$ also
    has infinitely many further elements, and hence all such worlds have truth value $t_\infty$ for
    $\varphi$. Therefore $\varphi$ cannot be contingent at such a $W'$. So whenever $\varphi$ is
    contingent at $W'$, the world $W'$ has only finitely many further elements.

    Since $\varphi$ is contingent at $W'$, there is an accessible world from $W'$ on which
    $\varphi$ has truth value different from $t_\infty$. Among the finitely many accessible finite
    sizes at which this happens, let $n$ be the greatest. Choose an accessible world $U$ from $W'$
    having exactly $n$ further elements outside the distinguished copy of $W$. Then either
    $\varphi$ or $\neg\varphi$ is penultimate at $U$: it is true at $U$, it becomes false at any
    accessible world with $n+1$ further elements, and it remains false at all further accessible
    worlds because every accessible world with more than $n$ further elements, and every accessible
    world with infinitely many further elements, has the eventual truth value $t_\infty$.

    Thus every contingency at every world in the cone above $W$ possibly attains a penultimate
    truth-value. Since $\theoryf{S4.3}$, and hence in particular $\theoryf{S4}$, is valid
    throughout the cone above $W$, the \nameref{Lowerboundstheorem2} yields that $\theoryf{Grz}$ is
    valid at every world in the cone above $W$.

    Finally, by the~\nameref{ConeLemma}, the normal closure of $\theoryf{S4.3}$ and
    $\theoryf{Grz}$, namely $\theoryf{Grz.3}$, is valid at $W$ with respect to
    ${\LDiamondofeq}_{W}$. By monotonicity of validities, the same lower bound holds for every
    language $\mathscr L$ with $\Lofeq\of\mathscr L\of {\LDiamondofeq}_{W}$.

    We next show that the propositional modal validities of $W$ are contained in \theoryf{Grz.3}.
    Let $N$ be an arbitrarily large natural number. For each $n<N$, let $r_n$ be the assertion that
    there are at least $n$ elements. We claim that the sequence $\langle r_n \mid |W|<n<N\rangle$
    is a ratchet that has not yet begun to crank at $W$. Indeed, all assertions $r_n$ in the
    sequence are false at $W$. Furthermore, since $W$ accesses precisely those worlds it injects to,
    the assertion $\necessary r_n$ is possible at $W$ (for instance, by moving to an $n$-element
    world), as well as $r_n \implies \necessary r_m$ for each $m<n$. By the~\nameref{Upperboundstheorem2},
    the propositional modal assertions valid at $W$ with respect to sentences, and therefore all
    assertions whatsoever, are contained in the propositional modal theory \theoryf{Grz.3}. Recall
    that the choice of $W$ was arbitrary, so the theorem holds for any finite world of
    $\Sets[$\into$]$.
\end{proof}

\section{The modal theory of the category of sets and bijections or identities}\label{Section.category-of-sets-with-isomorphisms}

For completeness, we finish by classifying validities of worlds that are subcategories of the category $\Sets[$\bij$]$ of sets and bijections. At this point, the observation below is immediate.

\begin{observation}\label{Observation.Iso-trivializations}
	The propositional modal validities of a world in any subcategory of the category $\Sets[$\bij$]$ of sets and bijections, with respect to any assertions whatsoever, in any first-order language $\mathscr{L}$---whether modal or not---that contains the first-order non-modal language $\Lofeq$ and is contained in the first-order modal language $\LDiamondofeq$, constitute precisely the trivial propositional modal theory \theoryf{Triv}. In other words, for any world $W$ in a Kripke category $\KripkeCat$ that is a subcategory of $\Sets[$\bij$]$, and any first-order language $\mathscr{L}$, modal or not, such that $\Lofeq \of \mathscr{L} \of {\LDiamondofeq}_{W}$, we have
	\begin{equation*}
		\Val_{\KripkeCat}(W, \mathscr{L}) = \theoryf{Triv}.
	\end{equation*}
\end{observation}

\section*{Acknowledgments}

I am grateful to Joel David Hamkins for his supervision and for many helpful conversations related to this work.

\printbibliography

\end{document}